\documentclass[final]{siamart250211}
\usepackage{amssymb,amsfonts}
\usepackage{bm}
\usepackage{algorithmicx}
\usepackage{algpseudocode}
\usepackage{enumitem}
\usepackage{textcomp}

\hypersetup{hypertexnames=false}

\numberwithin{equation}{section}

\numberwithin{theorem}{section}
\newsiamremark{remark}{Remark}

\newcommand{\be}{\begin{equation}}
\newcommand{\ee}{\end{equation}}
\newcommand{\ba}{\begin{aligned}}
\newcommand{\ea}{\end{aligned}}

\renewcommand{\Re}{\operatorname{Re}}
\renewcommand{\Im}{\operatorname{Im}}
\newcommand{\x}{\bm{x}}
\newcommand{\bk}{\bm{k}}
\newcommand{\thetastar}{\theta_{\star}}

\newcommand{\cstar}{c_{\star}}
\newcommand{\nexp}{M}

\makeatletter
\providecommand*{\theHALG@line}{}
\renewcommand*{\theHALG@line}{\thealgorithm.\arabic{ALG@line}}
\makeatother

\begin{document}

\title{A fast sum-of-Gaussians algorithm for the high-dimensional fractional Fokker--Planck equation}

\author{Shidong Jiang\thanks{Center for Computational Mathematics, Flatiron Institute, Simons Foundation, New York, NY 10010, USA (\email{sjiang@flatironinstitute.org}).}
\and Dong Wang\thanks{School of Science and Engineering, The Chinese University of Hong Kong, Shenzhen, Shenzhen, Guangdong 518172, P.~R.~China; Shenzhen International Center for Industrial and Applied Mathematics, Shenzhen Research Institute of Big Data, Shenzhen, Guangdong 518172, P.~R.~China (\email{wangdong@cuhk.edu.cn}).}
\and Qi Zhou\thanks{School of Mathematical Sciences, Shanghai Jiao Tong University, Shanghai 200240, P.~R.~China (\email{zhouqi1729@sjtu.edu.cn}).}}

\headers{Fast SOG algorithm for the high-dimensional FFPE}{S. Jiang, D. Wang, and Q. Zhou}

\maketitle

\begin{abstract}
We present a fast, high-order algorithm for the free-space fractional
Fokker--Planck equation (FFPE) in arbitrary spatial dimension. Its fundamental
solution, corresponding to a Dirac-delta initial condition, is obtained from the
explicit Fourier representation by applying a sum-of-Gaussians (SOG)
approximation to the nonseparable stretched exponential, using its complete
monotonicity as the Laplace transform of a one-sided $\alpha$-stable density.
Each Gaussian term is an ordinary heat kernel and therefore factorizes across
spatial coordinates. On a tensor-product grid, the separated form can be
assembled in $O(\nexp dN)$ work and storage, rather than forming all $O(N^d)$
grid values, where $\nexp$ is the number of Gaussian terms and $N$ is the
number of points per dimension. We prove an a~priori error estimate for the
pure-fractional fundamental solution and give a parameter-selection procedure
for prescribed accuracy over specified ranges of space and time. In numerical
experiments the method achieves more than ten digits of relative accuracy, with
$\nexp$ growing only logarithmically in the inverse tolerance, and maintains
this accuracy in dimensions up to $d=10^{5}$. This exceeds the dimensions
reached in comparable radial-quadrature tests, where the integrand
becomes increasingly oscillatory as the dimension grows. Because the method represents the
fundamental solution as a separated sum of heat kernels, any initial datum
given as a finite sum of tensor products can be evolved in closed form using
only one-dimensional convolutions. This yields a computable class of
high-dimensional solutions that is amenable to error analysis, and tensor neural
networks provide one possible way to construct such separated representations
for more general data.
\end{abstract}



\section{Introduction}
\label{sec::intro}

The Fokker--Planck equation (FPE) provides a deterministic description of the
time evolution of probability density functions for stochastic systems
\cite{chandrasekhar1943stochastic, risken1989fokker}, with applications in
statistical mechanics, stochastic processes, mathematical finance, information
theory, and machine learning
\cite{barato2015thermodynamic,black1973pricing,bressloff2014stochastic,ito2013information,mandt2017stochastic}.
In the classical regime, the underlying dynamics are typically driven by
Gaussian white noise, leading to the well-known linear growth of mean-squared
displacement, i.e., $\langle |x|^2 \rangle \propto t$
\cite{einstein1905theory}. However, many anomalous-transport models require
non-Gaussian jump statistics with algebraic tails and a self-similar length scale
that differs from the Brownian scale. In the spatially fractional model
considered here, the Fourier symbol $|\bm{k}|^{2\alpha}$ with $0<\alpha<1$
corresponds to a symmetric stable process of index $2\alpha$: the characteristic
length grows like $t^{1/(2\alpha)}$, while moments of order $q\ge 2\alpha$ are
infinite. Such heavy-tailed anomalous diffusion is modeled by the fractional
Fokker--Planck equation (FFPE), where the classical Laplacian is replaced by a
fractional Laplacian operator $(-\Delta)^\alpha$ \cite{metzler2000random},
defined for $\bm{x}\in \mathbb{R}^{d}$ through the Cauchy principal value
\begin{equation}
    \label{eq::fractional_operator}
    (-\Delta)^{\alpha}p(\bm{x})=\frac{2^{2\alpha}\Gamma(\alpha+d/2)}{\pi^{d/2}|\Gamma(-\alpha)|}\text{P.V.}\int_{\mathbb{R}^{d}}\frac{p(\bm{x})-p(\bm{y})}{|\bm{x}-\bm{y}|^{d+2\alpha}}\mathrm{d}\bm{y},
\end{equation}
where $\Gamma(z)$ denotes the Gamma function.

The fractional Laplacian $(-\Delta)^\alpha$ makes the equation nonlocal and
captures long-range jumps associated with L\'{e}vy stable processes
\cite{delia2020numerical,DuGunzburgerLehoucqZhou2012,DuGunzburgerLehoucqZhou2013,TianDu2013}. We consider the following initial value problem for
the high-dimensional FFPE:
\begin{equation}
\label{eq::FFPE_formula}
\begin{cases}
\frac{\partial }{\partial t}p(\bm{x},t) = -\bm{b} \cdot \nabla p(\bm{x},t) + D_{o}\Delta p(\bm{x},t) - D_{f}(-\Delta)^{\alpha}p(\bm{x},t), \\ 
p(\bm{x},0) = \delta(\bm{x}-\bm{x}_{0}), \quad \bm{x} \in \mathbb{R}^{d},
\end{cases}
\end{equation}
where $\bm{b} \in \mathbb{R}^d$ is the drift vector, and $D_o \ge 0$ and
$D_f > 0$ are the ordinary and fractional diffusion coefficients. The solution of
\cref{eq::FFPE_formula} is the \emph{fundamental solution} (Green's function) of
the FFPE; by linearity it determines, through convolution, the solution for a
general initial datum. We therefore take the Dirac-delta case as the basic
building block and use the same representation for separated initial data in
high dimension.

Numerical treatment of Eq.~\eqref{eq::FFPE_formula} in high dimensions poses
substantial mathematical and computational challenges. First, traditional
grid-based methods, such as finite difference or finite element schemes, suffer
from the curse of dimensionality: the degrees of freedom grow exponentially with
the dimension $d$ \cite{duo2018novel, han2018solving}. Second, the nonlocality of
$(-\Delta)^\alpha$ leads to dense discretized operators, which are costly for
large-scale problems. Monte Carlo sampling and deep-learning-based solvers
\cite{hu2025score, liu2022neural} avoid full grids, but the Dirac-delta initial
condition and slow convergence can still be limiting factors
\cite{han2018solving}. Recent methods based on functional hierarchical tensors
\cite{tang2024solving} and fundamental-solution integrals \cite{ye2026fast}
improve this situation. In particular, Ye et al.~\cite{ye2026fast} reduce the
free-space FFPE with Dirac-delta initial data to a one-dimensional radial
integral that is evaluated to high precision in low to moderate dimensions. That
approach relies on radial quadrature: as $d$ increases, the Bessel order
$(d-2)/2$ and the power $r^{d/2}$ in the integrand both grow, making the
integrand more oscillatory and increasing its dynamic range; the method is
demonstrated up to $d=29$. This motivates the separable representation developed
here, which is aimed at higher dimensions and at separated initial data. The two
approaches are complementary: the radial integral remains effective in low to
moderate dimension, whereas the present method is designed for high-dimensional
separated representations.

To address these challenges, we develop a fast algorithm based on a
sum-of-Gaussians (SOG) approximation of the fundamental solution. The fractional
operator enters the Fourier-space solution only through the stretched-exponential
factor $\exp(-D_ft|\bm{k}|^{2\alpha})$, the one piece that does not factorize
across coordinates. Because this factor is completely monotone
\cite{penson2010exact}, it is the Laplace transform of a one-sided
$\alpha$-stable density, and a trapezoidal discretization of that representation
turns it into a sum of Gaussians -- each an ordinary heat kernel that factorizes
across dimensions. The FFPE thus reduces to a short separated sum of decoupled
heat solutions. On a tensor-product grid, its one-dimensional factors are
assembled in $O(\nexp dN)$ work and storage rather than forming $O(N^d)$ dense
grid values, where $\nexp$ is the number of Gaussians and $N$ the number of
points per dimension.

This construction is accurate, admits a rigorous error analysis, and extends naturally to low-rank summation of separated initial data. We give a
rigorous a~priori error analysis for the pure-fractional kernel that fixes the
quadrature step and truncation for any prescribed tolerance, with convergence
governed by a complex-plane bound on the stable density. For prescribed physical
windows, the same scaled formulation gives a domain-adapted parameter choice for
all $\alpha\in(0,1)$ and includes the ordinary-diffusion case $D_o>0$; when
$\alpha=1/2$, a closed-form identity for the trapezoidal error provides an
additional analytic reference. In numerical
experiments the method attains more than ten digits of relative accuracy with $\nexp$
growing only logarithmically in the inverse tolerance, and sustains this accuracy
up to $d=10^{5}$, well beyond the dimensions reachable by
radial-quadrature methods. Because the relative
error obeys a self-similar scaling, a single approximation sized at the smallest
time serves an entire space-time window and avoids the small-time,
high-dimensional loss of accuracy observed for direct quadrature
\cite{ye2026fast}. Finally, since the method approximates the fundamental
solution by a separated sum of Gaussians, any initial datum written as a
sum of tensor products evolves in closed form through one-dimensional convolutions
alone -- a class of high-dimensional functions that is computable and amenable to
error analysis
\cite{beylkin2002numerical,beylkin2005algorithms,hackbusch2012tensor}.
Sparse-grid approximation methods \cite{ShenYu2010,ShenYu2012} and tensor
neural networks
\cite{wang2024tensor,wang2024posteriori,wang2024multieigenpairs}
provide complementary ways to construct reduced representations for more general
data.

Gaussian-sum approximations also arise in other high-dimensional PDE contexts.
For example, in time-independent many-electron Schr\"odinger eigenvalue problems,
mixed-derivative regularity \cite{Yserentant2004} provides analytic support for
sparse-grid approximations, and sparse-grid methods have been developed for the
Schr\"odinger equation \cite{GriebelHamaekers2007}. The pairwise Coulomb kernel
$1/|\bm{r}_i-\bm{r}_j|$, which is nonseparable in the electronic coordinates,
admits accurate SOG approximations; after expansion into Gaussian factors, it is
compatible with tensor-product integration and tensor neural network
representations \cite{Wu2026Spectral,ZhouWuLiuSunXieXu2025}. For high-dimensional evolution
problems, the same separability mechanism is relevant whenever a Fourier-space
propagator, or a linear subproblem arising from time discretization, admits an
accurate Gaussian-sum representation. For many nonlinear evolution equations, an
unconditionally energy-stable scalar auxiliary variable (SAV) temporal
discretization reduces each time step to a linear problem with known source terms
\cite{ShenXuYang2018}. If the resulting linear subproblem has a
constant-coefficient solution operator that admits an accurate and efficient
Gaussian-sum representation, the framework developed here extends naturally to
such problems. These connections motivate SOG approximations as building blocks
for separable representations in high-dimensional PDEs, although the analysis
below is restricted to the FFPE fundamental solution.

The remainder of this paper is organized as follows. In \cref{sec::prelim}, we review the mathematical preliminaries, including Fourier transforms, the theory of completely monotone functions, and properties of the stretched exponential function. In \cref{sec::method}, we detail the SOG algorithm and provide its rigorous error estimate. Numerical experiments assessing the performance of the proposed solver are presented in \cref{sec::Num}, followed by concluding remarks in \cref{sec::conclusion}.

\section{Preliminaries}
\label{sec::prelim}
\subsection{Fourier transform}
For a function $f\in L^{1}(\mathbb{R}^{d})\cap L^{2}(\mathbb{R}^{d})$ we define its \emph{Fourier transform} by
\be
\widehat{f}(\bk)=\mathcal{F}[f](\bk)=\int_{\mathbb{R}^{d}} f(\x)\,e^{-i\,\bk\cdot\x}\,\mathrm{d}\x,
\qquad \bk\in\mathbb{R}^{d},
\ee
and the \emph{inverse Fourier transform} by 
\be
f(\x)=\mathcal{F}^{-1}[\widehat{f}](\x)=\frac{1}{(2\pi)^d}\int_{\mathbb{R}^{d}} \widehat{f}(\bk)\,e^{i\,\bk\cdot\x}\,\mathrm{d}\bk,
\qquad \x\in\mathbb{R}^{d}.
\ee
If $f$ is \emph{radial}, i.e.\ $f(\x)=f(r)$ with $r=\sqrt{x_{1}^{2}+\dots+x_{d}^{2}}$, its Fourier transform is again radial. In this case, the \emph{radial Fourier transform} (i.e., the Hankel transform) pairs are
\be
\ba
\widehat{f}(k)&=\frac{(2\pi)^{d/2}}{k^{(d-2)/2}}
        \int_{0}^{\infty} f(r)\,r^{\frac{d}{2}} 
        J_{\frac{d-2}{2}}(kr)\,\mathrm{d}r,\\
f(r) &=\frac{1}{(2\pi)^{d/2} r^{(d-2)/2}}
        \int_{0}^{\infty} \widehat{f}(k)\,k^{\frac{d}{2}} 
        J_{\frac{d-2}{2}}(kr)\,\mathrm{d}k,       
\ea
\ee
where $J_{\nu}$ is the Bessel function of the first kind of order $\nu$.  
Finally, for a sufficiently regular and rapidly decaying $f$ the \emph{Poisson summation formula}~\cite{stein2011fourier} links sums over the integer lattice to sums over its dual:
\be
h\sum_{n=-\infty}^{\infty} f(hn + a) = \sum_{m=-\infty}^{\infty} e^{i \frac{2\pi m a}{h}} \widehat{f}\left(\frac{2\pi m}{h}\right),
\label{eq:poissonsf}
\ee
providing a powerful bridge between spatial and frequency-domain information.

\subsection{Completely monotone functions}
\begin{definition}[Completely monotone function]
A function \( f: (0, \infty) \to \mathbb{R} \) is \textbf{completely
monotone} if \(f\in C^\infty\) and
\[
(-1)^n f^{(n)}(x)\ge0
\]
for all nonnegative integers \(n\) and all \(x\in(0,\infty)\).
\end{definition} 

The following result, Bernstein's theorem, provides a crucial integral representation that is often used as an alternative definition
(see, for example, \cite{powell1981}).

\begin{lemma}[Bernstein's theorem]
A function \( f: (0, \infty) \to \mathbb{R} \) is completely monotone if and
only if it is the Laplace transform of a nonnegative Borel measure \( \mu \) on
\([0,\infty)\):
\[
f(x)=\int_0^\infty e^{-xt}\,\mathrm{d}\mu(t).
\]
If the measure has a density \( \rho(t)\ge0 \), this representation becomes
\[
f(x)=\int_0^\infty e^{-xt}\rho(t)\,\mathrm{d}t.
\]
\end{lemma}

\subsection{Properties of the stretched exponential and its inverse Laplace transform}
The function 
\be
f(x)=e^{-x^\alpha}, \qquad 0<\alpha<1, 
\label{eq:kwwfun}
\ee
also known as the stretched exponential or Kohlrausch--Williams--Watts (KWW) function, possesses several important properties.

\begin{lemma}
The stretched exponential function is completely monotone on $(0,\infty)$ and has the integral representation
\be
e^{-x^{\alpha}} = \int_0^{\infty} e^{-xt}\rho_\alpha(t)\,\mathrm{d}t,
\label{eq:kwwfunrep}
\ee
    where $\rho_\alpha$ is the probability density function (PDF) of a standard one-sided stable distribution (also called the one-sided L\'{e}vy $\alpha$-stable distribution). 
\end{lemma}

\begin{proof}
The fact that the stretched exponential function
is completely monotone can be shown via direct calculation of $f^{(n)}(x)$. The integral representation~\eqref{eq:kwwfunrep} can be 
found, say, in \cite{penson2010exact}.
\end{proof}
The following properties of $\rho_{\alpha}$
can be found in \cite{zolotarev1986}.
\begin{lemma}
\label{lemma::rho_alpha}
\begin{enumerate}
\item For $t>0$, $\rho_\alpha$ admits an integral
representation
\begin{equation}
    \label{eq::Levy}
    \rho_\alpha(t)=\frac{1}{\pi}\int_{0}^{+\infty}e^{-\cos(\alpha\pi)u^{\alpha}}e^{-ut}\sin(\sin(\alpha\pi)u^{\alpha})\mathrm{d}u.
\end{equation}
\item For $t > 1$, $\rho_\alpha$ admits the following series expansion
\be 
\rho_\alpha(t) = \frac{1}{\pi} \sum_{n=1}^{\infty} \frac{(-1)^{n-1}}{n!} \sin(\pi n \alpha) \Gamma(n\alpha+1) t^{-(n\alpha+1)}\le C_\alpha t^{-(1+\alpha)},
\label{eq:seriesexpansion}
\ee
where $C_\alpha$ is a positive constant depending on $\alpha$ (bounded on compact
subintervals of $(0,1)$).
\item As $t \to 0^+$, $\rho_\alpha$ has the 
asymptotic expansion:
\be 
\rho_\alpha(t) = C t^{\frac{\alpha-2}{2(1-\alpha)}} \exp\left(-D t^{-\frac{\alpha}{1-\alpha}}\right)\left( \sum_{k=0}^{\infty} a_k t^{\frac{k\alpha}{1-\alpha}} \right)
\label{eq:asympexpansion}
\ee 
where the constants $C$ and $D$ are positive and depend on $\alpha$: 
\be 
C = \frac{1}{\sqrt{2\pi (1-\alpha)}} \alpha^{\frac{1}{2(1-\alpha)}}, \qquad D = (1-\alpha) \alpha^{\frac{\alpha}{1-\alpha}},
\ee
and the first two coefficients $a_k$, $k=0,1$ are given by:
\be
a_0 = 1, \qquad 
a_1 = \frac{(2-\alpha)(1-2\alpha)}{24\alpha(1-\alpha)} \alpha^{-\frac{\alpha}{1-\alpha}}.
\ee
Moreover, the estimate
\begin{equation}
    \rho_\alpha(t)\le A_\alpha t^{-\gamma}\exp(-D t^{-\frac{\alpha}{1-\alpha}})
\end{equation}
holds with a constant $A_\alpha$ that has only an $O(1/\sqrt{1-\alpha})$ singularity as $\alpha\rightarrow 1^{-}$, where $\gamma=(2-\alpha)/(2-2\alpha)$.
\end{enumerate}    
\end{lemma}

\section{A fast sum-of-Gaussians FFPE solver}
\label{sec::method}
In this section, we present a fast algorithm for solving the FFPE~\eqref{eq::FFPE_formula}. By combining the radial Fourier-integral representation of the FFPE solution with a sum-of-Gaussians (SOG) approximation of the stretched exponential, the proposed method reduces the anomalous-diffusion solution to a sum of heat-equation solutions with closed-form expressions, thereby achieving a computational cost that grows linearly with the dimension.

\subsection{SOE approximation of the stretched exponential function}
\label{sec::SOE_approx}
We approximate the stretched exponential by a sum of exponentials (SOE),
\be
e^{-x^{\alpha}} \approx \sum_{\ell=-M_1}^{M_2} w_\ell e^{-s_\ell x}, \qquad x\in [\delta,R].
\label{eq:soeappr}
\ee
Our starting point is the integral representation
\eqref{eq:kwwfunrep}. 
Applying the change of variables $t=e^{u}$ to
\cref{eq:kwwfunrep}, we obtain
\be
e^{-x^{\alpha}} = \int_{-\infty}^{\infty}
e^{-x e^u+u}\rho_\alpha(e^u)\mathrm{d}u.
\label{eq:kwwfunrep2}
\ee
The integrand decays rapidly to zero as $u\to\pm\infty$, so the trapezoidal rule
converges exponentially fast~\cite{trefethen2014sirev}, and the nodes
$s_\ell$ and weights $w_\ell$ in \eqref{eq:soeappr} are given by
\be
s_\ell = e^{h\ell}, \qquad w_\ell = h s_\ell \rho_\alpha(s_\ell),
\ee
where $h>0$ is the step size in the trapezoidal rule. 

We analyze the approximation error of \cref{eq:soeappr}. We record the asymptotic magnitude of the Gamma function for a complex argument (Lemma~\ref{lemma:Gamma}, used in the parameter selection of \cref{sec::params}); its proof is given in the appendix of \cite{DEShaw2020JCP}.

\begin{lemma}\label{lemma:Gamma}
For fixed $x\in \mathbb{R}$, the Gamma function satisfies
\begin{equation}\label{eq::Gamma}
\left|\Gamma(x+iy)\right|\simeq (2\pi)^{1/2}(x^2+y^2)^{\frac{2x-1}{4}}e^{-\frac{\pi}{2}|y|}
\end{equation}
as $|y|\rightarrow \infty$.
\end{lemma}
We first analyze the discretization error
of the (infinite) trapezoidal rule
\begin{equation}
    \label{eq::SOE} 
\begin{aligned}
    e^{-x^{\alpha}}&\approx \sum_{\ell=-\infty}^{+\infty}w_\ell e^{-s_\ell x}.
\end{aligned} 
\end{equation}

\begin{theorem}
    \label{thm::fractional_approx}
    Let $\alpha\in(0,1)$ and $0<\epsilon_{\emph{SOE}}\le 1$, and set $\theta=-(1-\alpha)\theta_{*}$ with $\theta_{*}=\arctan(1/9)$. If the step size $h$ satisfies 
    \begin{equation}
    \label{eq::h_select}
        h\le \frac{2\pi(1-\alpha)\theta_*}{\log(1+2I_\alpha/\epsilon_{\emph{SOE}})}
    \end{equation}
    where 
    \begin{equation}
        I_{\alpha}:=\int_{0}^{\infty}|\rho_\alpha(re^{i\theta})|\mathrm{d}r,
    \end{equation}
    then 
    \begin{equation}
        \label{eq::error}
        \left|e^{-x^{\alpha}}-\sum_{\ell=-\infty}^{+\infty}w_\ell e^{-s_\ell \cdot x}\right|\le \epsilon_{\emph{SOE}},
    \end{equation}
    holds for all $x\in [0,+\infty)$.
\end{theorem}

\begin{proof}
Combining the Poisson summation formula~\eqref{eq:poissonsf} and the fact that $\widehat{f}(0) = \int_\mathbb{R} f(u)du$, we obtain
\begin{equation}
\label{eq::poisson}
\left|\int_{\mathbb{R}}f(u)\mathrm{d}u-h\sum_{n\in \mathbb{Z}}f(nh)\right|\le \sum_{m\neq 0}\left|\widehat{f}\left(\frac{2\pi m}{h}\right)\right|,
\end{equation}
We now apply the above inequality to the integral representation~\eqref{eq:kwwfunrep2} of the stretched exponential, i.e.,
$f(u)=e^{u}\rho_\alpha(e^{u})e^{-xe^{u}}$. The error bound on the right-hand side of \eqref{eq::poisson} is determined by the decay rate of $\widehat{f}$. We have
\begin{equation}
\label{eq::f_Fourier}
\begin{aligned}
\widehat{f}(k)&=\int_{-\infty}^{+\infty}e^{u}\rho_\alpha(e^{u})e^{-xe^{u}}e^{-iku}\mathrm{d}u\\
&=\int_{0}^{+\infty}e^{-xt}\rho_{\alpha}(t)t^{-ik}\mathrm{d}t.\\
\end{aligned}
\end{equation}
        Let us consider the asymptotic approximation of the integral in \cref{eq::f_Fourier}. We transform the integral path to $C_\theta=\{re^{i\theta},r\in[0,+\infty)\}$. By Cauchy's theorem, for a complex function $g(z)$, if the angular region between $C_0$ and $C_\theta$ contains no singularity, then
        \begin{equation}
            \int_{C_0}g(z)\mathrm{d}z=\int_{C_\theta}g(z)\mathrm{d}z.
        \end{equation}
        Since $|\theta|\le \pi(1-\alpha)/2$, one then derives that
        \begin{equation}
            \label{eq::Int_trans_calc}
            \begin{aligned}
\widehat{f}(k)
&=\int_{0}^{+\infty}e^{-xt}\rho_{\alpha}(t)t^{-ik}\mathrm{d}t,\\
&=\int_{0}^{+\infty}e^{-xre^{i\theta}}\rho_{\alpha}(re^{i\theta})r^{-ik}e^{\theta k}\mathrm{d}r\\
&=e^{\theta k}\int_{0}^{+\infty}e^{-xre^{i\theta}}\rho_{\alpha}(re^{i\theta})r^{-ik}\mathrm{d}r.
\end{aligned}
        \end{equation}
        We choose the rotation opposite to the sign of $k$, i.e. $\theta=-\operatorname{sign}(k)\,\theta_{*}(1-\alpha)$, which lies in the sector of analyticity since $\theta_{*}(1-\alpha)<\pi(1-\alpha)/2$. Because $x\ge 0$ and $\cos\theta>0$ give $|e^{-xre^{i\theta}}|=e^{-xr\cos\theta}\le 1$, and $|r^{-ik}|=1$, the rotated integral in \cref{eq::Int_trans_calc} is bounded in modulus by $\int_{0}^{\infty}|\rho_{\alpha}(re^{i\theta})|\mathrm{d}r$. Moreover, since $\rho_\alpha$ is real on the positive axis, the reflection $\rho_\alpha(\bar z)=\overline{\rho_\alpha(z)}$ shows that this integral takes the same value $I_\alpha$ for $+\theta$ and $-\theta$. With $e^{\theta k}=e^{-(1-\alpha)\theta_{*}|k|}$ for this sign choice, one deduces the convergence rate with respect to mode $k$ that
        \begin{equation}
            \label{eq::conv_rate_k}
            |\widehat{f}(k)|\le I_{\alpha} e^{-(1-\alpha)\theta_{*} |k|}.
        \end{equation}
        Substituting \cref{eq::conv_rate_k} into \cref{eq::poisson}, one has
        \begin{equation}
            \label{eq::approx_error}
            \begin{aligned}
                \left|\int_{\mathbb{R}}f(u)\mathrm{d}u-h\sum_{n\in \mathbb{Z}}f(nh)\right|&\le \sum_{m\neq 0} I_\alpha e^{-2\pi(1-\alpha)\theta_{*}|m|/h}\\
                &\le 2I_\alpha e^{-2\pi(1-\alpha)\theta_{*}/h}\frac{1}{1-e^{-2\pi(1-\alpha)\theta_{*}/h}}\\
                &\le \epsilon_{\text{SOE}}.
            \end{aligned}
        \end{equation}
 \end{proof}
Theorem~\ref{thm::fractional_approx} bounds the discretization error of the
\emph{infinite} rule uniformly on $[0,+\infty)$. Truncating the series to
$\ell\in[-M_1,M_2]$, as in \cref{eq:soeappr}, restricts the accurate range to
the finite window $[\delta,R]$, whose endpoints are governed by the extreme
retained nodes: since a single term $e^{-s_\ell x}$ acts on the scale $x\sim
s_\ell^{-1}$, the largest node $s_{M_2}=e^{hM_2}$ sets the lower limit
$\delta\sim s_{M_2}^{-1}$ and the smallest node $s_{-M_1}=e^{-hM_1}$ the upper
limit $R\sim s_{-M_1}^{-1}$. The truncation indices $M_1,M_2$ are chosen in
\cref{sec::Error} to meet the target tolerance.

Theorem \ref{thm::I_alpha_bound} provides a uniform bound of $I_\alpha$.
\begin{theorem}
    \label{thm::I_alpha_bound}
    For $\alpha\in(0,1)$ and
    $\theta=\pm(1-\alpha)\theta_{*}$, there is a universal constant $C_I$,
    independent of $\alpha$, such that
    \begin{equation}
        \label{eq::I_alpha_upper}
        I_\alpha=\int_0^\infty |\rho_\alpha(re^{i\theta})|\,\mathrm{d}r
        \le C_I\left(\frac{1}{\alpha}+\frac{1}{1-\alpha}\right).
    \end{equation}
    More precisely, if $I_\alpha^1$ and $I_\alpha^2$ denote the contributions
    from $r\in(0,1)$ and $r\in(1,\infty)$, then
    \[
    I_\alpha^1\le
    \begin{cases}
    C_I/\alpha, & 0<\alpha\le 1/2,\\[2mm]
    C_I\log\!\bigl(e/(1-\alpha)\bigr), & 1/2\le \alpha<1,
    \end{cases}
    \qquad
    I_\alpha^2\le C_I\left(\frac{1}{\alpha}+\frac{1}{1-\alpha}\right).
    \]
\end{theorem}
The detailed proof is provided in Appendices~\ref{sec::rho_estimate}
and~\ref{sec::int_bound}. The
estimate in \cref{eq::I_alpha_upper} shows only mild endpoint singularities. The
large factor that appears in a direct pointwise bound for
$|\rho_\alpha(z)|$ near $\alpha=1$ is not intrinsic to $I_\alpha$; in the proof
of \cref{sec::int_bound} the compact part is integrated in the radial variable
before the saddle-contour integral is estimated, which preserves the
cancellation in the phase. At $\alpha=1$ the representing measure degenerates to
a Dirac mass at $z=1$, while the limit $\alpha\to0^+$ is also degenerate and is
no longer described by a regular probability density on $(0,\infty)$.
Accordingly, endpoint regimes still require care in numerical density
evaluation, but the contour constant entering the trapezoidal error does not
grow like $c^{1/(1-\alpha)}$.

\subsection{The SOG approximation of the solution} Taking the Fourier transform of \cref{eq::FFPE_formula} with respect to the space variable $\bm{x}\in \mathbb{R}^d$, we obtain
\begin{equation}
    \label{eq::FFPE_Fourier}
    \left\{\begin{aligned}
&\frac{\partial}{\partial t}\widehat{p}(\boldsymbol{k}, t) =-i(\bm{k}\cdot\bm{b}) \widehat{p}(\boldsymbol{k}, t)-D_{\mathrm{o}}|\boldsymbol{k}|^2 \widehat{p}(\boldsymbol{k}, t)-D_{\mathrm{f}}|\boldsymbol{k}|^{2 \alpha} \widehat{p}(\boldsymbol{k}, t) \\
&\widehat{p}(\boldsymbol{k}, 0) =\exp \left(-i\bm{k}\cdot\bm{x}_0\right),
\end{aligned}\right.
\end{equation}
which can be explicitly solved with
\begin{equation}
    \label{eq::p_hat_explicit}
    \widehat{p}(\boldsymbol{k}, t)=e^{-i\bm{k}\cdot\bm{x}_0^t}e^{-(D_o|\bm{k}|^2+D_f|\bm{k}|^{2\alpha})t},\quad \bm{x}_0^t:=\bm{x}_0+\bm{b}t.
\end{equation}
By the inverse Fourier transform, we have
\begin{equation}
    \label{eq::p_intergral}
    p(\bm{x},t)=\frac{1}{(2\pi)^d}\int_{\mathbb{R}^d}e^{i\bm{k}\cdot(\bm{x}-\bm{x}_0^t)}e^{-(D_o|\bm{k}|^2+D_f|\bm{k}|^{2\alpha})t}\mathrm{d}\bm{k}.
\end{equation}
Since the factor $\exp(-(D_o|\bm{k}|^2+D_f|\bm{k}|^{2\alpha})t)$ is radially symmetric and hence invariant under coordinate rotations, we choose a new orthogonal coordinate frame $(\bm{e}_1,\cdots,\bm{e}_d)$ such that $\bm{x}-\bm{x}_0^{t}$ lies along the first coordinate axis, that is, $\bm{x}-\bm{x}_0^{t}=y\bm{e}_1$ with the direction vector $|\bm{e}_1|=1$. For $d\ge2$ and $y>0$, by performing a $d$-dimensional spherical coordinate transformation $(k_1,k_2,\cdots,k_d)\rightarrow (r,\theta_1,\theta_2,\cdots,\theta_{d-1})$, the integral in \cref{eq::p_intergral} is equivalent to
\begin{equation}
    \label{eq::p_int_equiv}
    p(\bm{x},t)=\frac{\int_0^\pi \sin ^{d-2}(\theta) \int_0^{\infty} r^{d-1} \cos (\cos (\theta) y r) e^{-(D_or^2+D_f r^{2\alpha})t} \mathrm{d}r \mathrm{d} \theta}{2^{d-1}\pi^{\frac{d+1}{2}}\Gamma(\frac{d-1}{2})}
\end{equation}
where $\Gamma(\cdot)$ denotes the Gamma function, and $\theta$ abbreviates the first angular coordinate $\theta_1$. Performing the angular integration for $y\neq0$, and evaluating \cref{eq::p_intergral} directly for $y=0$, gives the explicit radial integral expression~\cite{ye2026fast}
\begin{equation}
    \label{eq::int_expression}
    p(\bm{x},t)=
    \left\{
\begin{aligned}
    &\frac{1}{y^{\frac{d-2}{2}}}\int_0^{\infty}\left(\frac{r}{2\pi}\right)^{\frac{d}{2}}J_{\frac{d-2}{2}}(yr)\exp(-(D_or^2+D_fr^{2\alpha})t)\mathrm{d}r,\ y\neq 0\\
   &\frac{2^{1-d}}{\pi^{d/2}\Gamma\left(d/2\right)}\int_0^{\infty}r^{d-1}\exp(-(D_or^2+D_fr^{2\alpha})t)\mathrm{d}r,\ y=0,
\end{aligned}
\right.
\end{equation}
where $J_\nu(\cdot)$ represents the $\nu$-th Bessel function of the first kind
\begin{equation}
\label{eq::Bessel}
J_\nu(z)=\frac{(z/2)^{\nu}}{\pi^{1/2}\Gamma(\nu+1/2)}\int_0^{\pi}\sin^{2\nu}(\theta)\cos(\cos(\theta)z)\mathrm{d}\theta.
\end{equation}
For a fixed terminal time $T > 0$ and $t\in[0,T]$, we apply the SOE expansion~\eqref{eq:soeappr} to the fractional factor in $\widehat{p}(\bm{k},t)$, namely $e^{-D_ft|\bm{k}|^{2\alpha}}=e^{-x^{\alpha}}$ with $x=(D_ft)^{1/\alpha}|\bm{k}|^2$, giving the sum-of-Gaussians (SOG) approximation
\begin{equation}  \label{eq::SOG_approx}
    \widehat{p}(\boldsymbol{k}, t)\approx\widehat{p}_{\text{SOG}}(\bm{k},t)=\sum_{\ell=-M_1}^{M_2}w_\ell e^{-i\bm{k}\cdot\bm{x}_0^t}e^{-[D_ot+s_\ell(D_ft)^{1/\alpha}]|\bm{k}|^2}
\end{equation}
In real space, this procedure is equivalent to approximating the true solution $p(\bm{x}, T)$ through a linear superposition of ordinary diffusion kernels. Specifically, this approximation is formulated as
\begin{equation}
    \label{eq::p_approx}
    p(\bm{x},T)\approx p_{\text{SOG}}(\bm{x},T):=\sum_{\ell=-M_1}^{M_2}w_\ell p_\ell(\bm{x},T),
\end{equation}
where $p_\ell(\bm{x},T)$ is the solution to the following heat equation
\begin{equation}
    \label{eq::heat}
    \left\{
\begin{aligned}
    \frac{\partial}{\partial t}p_\ell(\bm{x},t)&=-\bm{b}\cdot\nabla p_\ell(\bm{x},t)+\left(D_{o}+\frac{s_\ell D_f}{\alpha} (D_ft)^{\frac{1}{\alpha}-1}\right)\Delta p_\ell(\bm{x},t)\\
    p_\ell(\bm{x},0)&=\delta_{\bm{x}_0}(\bm{x}).
\end{aligned}
\right.
\end{equation}
The exact solution of \cref{eq::heat} reads
\begin{equation}
    \label{eq::Heat_explicit}
    p_\ell(\bm{x},T)=\frac{1}{(4\pi C_\ell^{T})^{d/2}}\exp\left(-\frac{y^2}{4C_\ell^{T}}\right),
\end{equation}
where $C_\ell^{T}:=D_oT+s_\ell (D_fT)^{1/\alpha}$ is the $\ell$-th ordinary diffusion constant in the SOG approximation.

Therefore, on a tensor-product observation grid
$\bm{X}=\otimes_{j=1}^d X_j$, each kernel $p_\ell$ in
\cref{eq::Heat_explicit} is represented by its one-dimensional Gaussian
factors. Assembling the factors for one term costs
$O(\sum_{j=1}^d n_j)$ work and storage, and assembling all factors costs
$O(\nexp\sum_{j=1}^d n_j)$, or $O(\nexp dN)$ when $n_j=N$ for all $j$.
This is the cost of the separated representation; forming and storing all
dense values on the full tensor grid would still require $O(N^d)$ entries.

\subsection{Error estimate of the SOG approximation}
\label{sec::Error}

We now analyze the SOG approximation error of the FFPE solution given by \cref{eq::p_approx}. Denote
\begin{equation}
    \mathcal{E}_T=\|p(\bm{x},T)-p_{\text{SOG}}(\bm{x},T)\|_{\infty}
\end{equation}
as the $L_\infty$ error at the fixed terminal time $T$. By the inverse Fourier transform, $\mathcal{E}_T$ admits the simple bound
\begin{equation}
\label{eq::IFT_error}
\mathcal{E}_T\le \frac{1}{(2\pi)^d}\int_{\mathbb{R}^d}|\widehat{e}(\bm{k},T)|\mathrm{d}\bm{k},\quad \widehat{e}(\bm{k},T):=\widehat{p}(\bm{k},T)-\widehat{p}_{\text{SOG}}(\bm{k},T).
\end{equation}
Using the explicit expression in \cref{eq::p_hat_explicit} and the invariance of Gaussians under the Fourier transform, we obtain
\begin{equation}
\label{eq::IFT_error_explicit}
\mathcal{E}_T\le \frac{1}{(2\pi)^d}\int_{\mathbb{R}^d}\left|e^{-D_oT|\bm{k}|^2}\left(e^{-D_fT|\bm{k}|^{2\alpha}}-\sum_{\ell=-M_1}^{M_2}w_\ell e^{-s_\ell(D_fT)^{1/\alpha}|\bm{k}|^2}\right)\right|\mathrm{d}\bm{k}.
\end{equation}
In fact, the error $\mathcal{E}_T$ stems from the trapezoidal discretization of the kernel integral and the truncation of the infinite series at both ends; hence, we decompose it into three parts
\begin{equation}
    \label{eq::error_split}
    \mathcal{E}_T\le \mathcal{E}_\infty+\mathcal{E}_{\text{up}}+\mathcal{E}_{\text{down}},
\end{equation}
where
\begin{equation}
\begin{aligned}
    \mathcal{E}_\infty&=\frac{1}{(2\pi)^d}\int_{\mathbb{R}^d}\left|e^{-D_oT|\bm{k}|^2}\left(e^{-D_fT|\bm{k}|^{2\alpha}}-\sum_{\ell=-\infty}^{\infty}w_\ell e^{-s_\ell(D_fT)^{1/\alpha}|\bm{k}|^2}\right)\right|\mathrm{d}\bm{k},\\
    \mathcal{E}_{\text{up}}&=\frac{1}{(2\pi)^d}\int_{\mathbb{R}^d}\left|e^{-D_oT|\bm{k}|^2}\sum_{\ell=M_2+1}^{\infty}w_\ell e^{-s_\ell(D_fT)^{1/\alpha}|\bm{k}|^2}\right|\mathrm{d}\bm{k},\\
    \mathcal{E}_{\text{down}}&=\frac{1}{(2\pi)^d}\int_{\mathbb{R}^d}\left|e^{-D_oT|\bm{k}|^2}\sum_{\ell=-\infty}^{-M_1-1}w_\ell e^{-s_\ell(D_fT)^{1/\alpha}|\bm{k}|^2}\right|\mathrm{d}\bm{k}.\\
\end{aligned}
\end{equation}
Clearly, when the ordinary diffusion is absent, i.e., $D_o=0$, controlling the
error is most challenging because there is no additional Gaussian damping.
We therefore prove the conservative a~priori estimates in this worst case. When
$D_o>0$, the same scaled representation is used, with the extra parameter
$\lambda(t)$ retained over the prescribed time window in \cref{sec::params}.

For $\mathcal{E}_\infty$, since the integrand is radially symmetric, we switch to $d$-dimensional spherical coordinates and set $u=(D_fT)^{1/\alpha}|\bm{k}|^2$. Then the integral in \cref{eq::IFT_error_explicit} can be reduced to a one-dimensional integral with respect to $u$, namely,
\begin{equation}
    \label{IFT_1d}
    \mathcal{E}_{\infty}= \frac{\Omega_d}{2(2\pi)^d}(D_fT)^{-\frac{d}{2\alpha}}\int_{0}^{\infty}|\epsilon_\infty(u)|u^{\frac{d}{2}-1}\mathrm{d}u,
\end{equation}
where $\Omega_d=2\pi^{d/2}/\Gamma(d/2)$ denotes the surface area of the unit sphere in $\mathbb{R}^d$, and
\begin{equation}
    \epsilon_\infty(u):=e^{-u^{\alpha}}-\sum_{\ell=-\infty}^{+\infty}w_\ell e^{-s_\ell u}
\end{equation}
denotes the approximation error of the kernel function for $u\in [0,+\infty)$, which is bounded pointwise by $|\epsilon_\infty(u)|\le \epsilon_{\text{SOE}}$ under the parameter choice of Theorem~\ref{thm::fractional_approx}. This uniform bound is useful only on a finite interval; it cannot by itself be integrated over $[U,\infty)$ against the growing measure factor $u^{d/2-1}$. We therefore split the integral in \cref{IFT_1d} over $[0,U]$ and $[U,+\infty)$, and bound the tail directly. Let
\begin{equation}
\label{eq::tail_infty}
\begin{aligned}
\mathcal T_h(U)={}&\int_U^\infty e^{-u^\alpha}u^{\frac d2-1}\,\mathrm du\\
&+\sum_{\ell=-\infty}^{\infty}w_\ell
\int_U^\infty e^{-s_\ell u}u^{\frac d2-1}\,\mathrm du  \\
={}&\frac{1}{\alpha}\Gamma\!\left(\frac{d}{2\alpha},U^\alpha\right)
+\sum_{\ell=-\infty}^{\infty}w_\ell s_\ell^{-\frac d2}
\Gamma\!\left(\frac d2,s_\ell U\right),
\end{aligned}
\end{equation}
where $\Gamma(a,z)$ is the upper incomplete Gamma function. Then
\begin{equation}
    \label{eq::E_infty_estimate}
\mathcal{E}_{\infty}\le \frac{1}{2^d\pi^{d/2}\Gamma(d/2)}(D_fT)^{-\frac{d}{2\alpha}}\left[\frac{2\epsilon_{\text{SOE}}}{d}U^{\frac{d}{2}}+\mathcal T_h(U)\right].
\end{equation}

Next, we analyze the error bounds for $\mathcal{E}_{\text{up}}$ and $\mathcal{E}_{\text{down}}$. By applying the same change of variables as for $\mathcal{E}_\infty$, the exponential integrals can be evaluated in closed form. Indeed, let
\begin{equation}
\begin{aligned}
     I_\ell&:=w_\ell \int_{0}^{\infty}e^{-s_\ell u}u^{\frac{d}{2}-1}\mathrm{d}u=h\Gamma\left(d/2\right)s_\ell^{1-\frac{d}{2}}\rho_\alpha(s_\ell),
\end{aligned}
\end{equation}
and $\mathcal{E}_{\text{up}}$ is then estimated using property~2 of Lemma~\ref{lemma::rho_alpha}, which gives
\begin{equation}
    \label{eq::err_up}
    \begin{aligned}
        \mathcal{E}_{\text{up}}&=\frac{1}{2^d\pi^{d/2}\Gamma(d/2)}(D_fT)^{-\frac{d}{2\alpha}}\sum_{\ell=M_2+1}^{\infty}I_\ell\\
        &= \frac{1}{2^d\pi^{d/2}\Gamma(d/2)}(D_fT)^{-\frac{d}{2\alpha}}\sum_{\ell=M_2+1}^{\infty}h\Gamma\left(d/2\right)s_\ell^{1-\frac{d}{2}}\rho_\alpha(s_\ell)\\
        &\le \sum_{\ell=M_2+1}^{\infty}\frac{h}{2^d\pi^{d/2}}(D_fT)^{-\frac{d}{2\alpha}}s_\ell^{1-\frac{d}{2}}\cdot C_\alpha s_\ell^{-(1+\alpha)}\\
        &=\frac{hC_\alpha (D_fT)^{-\frac{d}{2\alpha}}}{2^d\pi^{d/2}(e^{h(\alpha+d/2)}-1)}e^{-h(\alpha+d/2)M_2}.
    \end{aligned}
\end{equation}
For $\mathcal{E}_{\text{down}}$ we invoke the asymptotic estimate of $\rho_\alpha(t)$ as $t\to 0^{+}$ (property~3 of Lemma~\ref{lemma::rho_alpha}):
\begin{equation}
    \label{eq::err_down}
     \begin{aligned}
     \mathcal{E}_{\text{down}}&=\frac{1}{2^d\pi^{d/2}\Gamma(d/2)}(D_fT)^{-\frac{d}{2\alpha}}\sum_{\ell=-\infty}^{-M_1-1}I_\ell\\
     &=\frac{1}{2^d\pi^{d/2}\Gamma(d/2)}(D_fT)^{-\frac{d}{2\alpha}}\sum_{\ell=-\infty}^{-M_1-1}h\Gamma\left(d/2\right)s_\ell^{1-\frac{d}{2}}\rho_\alpha(s_\ell)\\
     &\le \sum_{\ell=M_1+1}^{\infty}\frac{h (D_fT)^{-\frac{d}{2\alpha}}}{2^d\pi^{d/2}}\cdot A_\alpha e^{h\ell(\gamma+\frac{d}{2}-1)}\exp\big(-D e^{h\ell\frac{\alpha}{1-\alpha}}\big)\\
     &\le \frac{C_{0}(D_fT)^{-\frac{d}{2\alpha}}}{2^d\pi^{d/2}}\cdot hA_\alpha e^{hM_1(\gamma+\frac{d}{2}-1)}\exp\big(-D e^{hM_1\frac{\alpha}{1-\alpha}}\big),\\
     \end{aligned}
\end{equation}
Here $C_0$ denotes a tail-majorization constant after $M_1$ is chosen beyond
the maximizer of
$e^{h\ell(\gamma+d/2-1)}\exp\{-D e^{h\ell\alpha/(1-\alpha)}\}$; from that
point the summand is decreasing, and the remaining lower tail is bounded by a
fixed multiple of the displayed value at $M_1$.

Next, based on the error bounds in \cref{eq::E_infty_estimate,eq::err_up,eq::err_down}, we provide a rigorous parameter-selection strategy. Given a prescribed target tolerance $\epsilon$, split it into budgets
$\epsilon_\infty+\epsilon_{\mathrm{up}}+\epsilon_{\mathrm{down}}\le\epsilon$.
For the infinite-rule part, choose $U$ and $\epsilon_{\text{SOE}}$ so that
\begin{equation}
    \label{eq::para_U}
    \mathcal T_h(U)\le
    2^{d-1}\pi^{\frac d2}\Gamma(d/2)(D_fT)^{\frac{d}{2\alpha}}\epsilon_\infty .
\end{equation}
The SOE tolerance $\epsilon_{\text{SOE}}$ is chosen to satisfy
\begin{equation}
    \label{eq::para_SOE}
    \epsilon_{\text{SOE}}\le \frac{d\epsilon_\infty}{4}\cdot(D_fT)^{\frac{d}{2\alpha}}\cdot 2^d\pi^{\frac{d}{2}}\Gamma(d/2)\cdot U^{-\frac{d}{2}},
\end{equation}
which in turn fixes the trapezoidal step size $h$ via \cref{eq::h_select}. The tail condition~\eqref{eq::para_U} is then checked with this $h$; if necessary, $U$ and $\epsilon_{\text{SOE}}$ are adjusted iteratively. Finally, we determine the SOG truncation parameters $M_1$ and $M_2$. For the upper truncation $M_2$, the selection becomes
\begin{equation}
    \label{eq::para_M2}
    M_2\ge \frac{1}{h(\alpha+\frac{d}{2})}\log\left[\frac{hC_\alpha (D_fT)^{-\frac{d}{2\alpha}}}{2^d\pi^{\frac{d}{2}}\epsilon_{\mathrm{up}}(e^{h(\alpha+\frac{d}{2})}-1)}\right].
\end{equation}
For the lower truncation we take $M_1=\log X/h$, where $X$ satisfies
\begin{equation}
    \label{eq::para_M1}
    X^{\gamma+\frac{d}{2}-1}\exp\left(-DX^{\frac{\alpha}{1-\alpha}}\right)\le \frac{2^d\pi^{\frac{d}{2}}\epsilon_{\mathrm{down}}\cdot (D_fT)^{\frac{d}{2\alpha}}}{C_0 A_\alpha\cdot h},
\end{equation}
which provides the asymptotic bound
\begin{equation}
    \label{eq::para_M1_asymp}
    M_1\ge \frac{1}{h}\log\left[X_0+\frac{(1-\alpha)(\gamma+\frac{d}{2}-1)}{D\alpha}X_0^{\frac{1-2\alpha}{1-\alpha}}\log X_0\right]
\end{equation}
with
\begin{equation}
    \label{eq::X_0}
    X_0=\left[\frac{1}{D}\log\left(\frac{C_0A_\alpha h }{2^d\pi^{\frac{d}{2}}\epsilon_{\mathrm{down}}\cdot (D_fT)^{\frac{d}{2\alpha}}}\right)\right]^{\frac{1-\alpha}{\alpha}}.
\end{equation}

\begin{remark}[Self-similarity and physical time windows]
\label{rmk::selfsim}
For $D_o=0$, the Green's function has the self-similar form
$p(Y,T)=(D_fT)^{-d/2\alpha}Q_{\alpha,d}(\eta)$ with
$\eta=|Y|/(D_fT)^{1/(2\alpha)}$, the $\lambda=0$ special case of
\cref{eq::scaled_q_integral}. Thus, if the evaluation set is specified in the
scaled coordinate $\eta$, the relative quadrature problem is independent of the
terminal time after the common factor $(D_fT)^{-d/2\alpha}$ is removed. This is
the precise sense in which the pure-fractional kernel is time-self-similar.

It does \emph{not} mean that a fixed physical window is independent of time. If
$0\le |Y|\le R$ and $T\in[t_{\min},t_{\max}]$, the scaled interval is
$0\le\eta\le\eta_{\max}$ with
$\eta_{\max}=R/(D_ft_{\min})^{1/(2\alpha)}$. Decreasing $t_{\min}$ or increasing
$R$ enlarges the active range of Gaussian scales and may increase the number of
terms. Once the SOG approximation has been constructed for this largest scaled radius,
the same nodes and weights are reused for all later times in the interval. When
$D_o>0$, after factoring out the common scale
$(D_ft)^{-d/(2\alpha)}$, the remaining dimensionless kernel also depends on
$\lambda(t)=D_ot/(D_ft)^{1/\alpha}$; in \cref{eq::scaled_q_integral}, this
parameter enters only through the shift $s\mapsto s+\lambda(t)$.
Equivalently, in the Fourier-side error formula \cref{eq::IFT_error_explicit},
the SOG error is multiplied by $e^{-D_ot|\bm{k}|^2}\le1$. Thus, for the
absolute-error estimate used here, the pure-fractional case $D_o=0$ is the
worst case; positive ordinary diffusion can only add Gaussian damping.
\end{remark}

\subsection{Parameter selection for a prescribed tolerance and evaluation window}
\label{sec::params}
The estimates in \cref{sec::SOE_approx,sec::Error} give an all-$\alpha$
a~priori construction for the scalar multiplier $e^{-x^\alpha}$ and its induced
Green's-function approximation. This subsection formulates the corresponding
parameter choice directly at the level of the scaled Green's function. The
construction applies for every $\alpha\in(0,1)$; the special value
$\alpha=1/2$ enters only through the closed-form identities recorded at the end
of the subsection.

\medskip\noindent\textbf{Scaled Green's-function representation.}\;
Let $y=|\bm{x}-\bm{x}_0-\bm{b}t|$ and define
\begin{equation}\label{eq::scaled_profile}
    \eta=\frac{y}{(D_ft)^{1/(2\alpha)}},\qquad
    \lambda(t)=\frac{D_ot}{(D_ft)^{1/\alpha}}.
\end{equation}
The following proposition separates the self-similar scaling from the
finite Gaussian quadrature and transfers relative-error estimates from the
scaled profile to the physical Green's function.
\begin{proposition}[Scaled representation and relative error]
\label{prop::Q_sog}
Let $\alpha\in(0,1)$, $d\ge1$, $D_f>0$, $D_o\ge0$, and $t>0$. For
$Y=\bm{x}-\bm{x}_0-\bm{b}t$, $y=|Y|$, define $\eta$ and $\lambda(t)$ by
\cref{eq::scaled_profile}. Then the Green's function has the exact scaled
representation
\begin{equation}\label{eq::scaled_q}
    p(y,t)=(D_ft)^{-\frac{d}{2\alpha}}Q_{\alpha,d}(\eta,\lambda(t)),
\end{equation}
where
\begin{equation}\label{eq::scaled_q_integral}
    Q_{\alpha,d}(\eta,\lambda)
    =\int_0^\infty \rho_\alpha(s)\,[4\pi(s+\lambda)]^{-d/2}
    \exp\!\left(-\frac{\eta^2}{4(s+\lambda)}\right)\,\mathrm ds .
\end{equation}
The finite SOG rule with $s_\ell=e^{\ell h}$ and
$w_\ell=h\,s_\ell\rho_\alpha(s_\ell)$ gives
\begin{equation}\label{eq::scaled_q_sog}
    Q_{h,L,U}(\eta,\lambda)
    =\sum_{\ell=L}^{U}w_\ell\,[4\pi(s_\ell+\lambda)]^{-d/2}
    \exp\!\left(-\frac{\eta^2}{4(s_\ell+\lambda)}\right).
\end{equation}
Consequently, the finite physical-space approximation
\[
    p_{h,L,U}(y,t)=(D_ft)^{-\frac{d}{2\alpha}}
    Q_{h,L,U}(\eta,\lambda(t))
\]
is exactly the finite sum of Gaussian heat kernels in
\cref{eq::Heat_explicit}. Moreover, for any physical window
$\mathcal W$ and its scaled image
\[
    \mathcal D_{\mathcal W}
    =\{(\eta,\lambda(t)):\ (y,t)\in\mathcal W\},
\]
the pointwise relative errors are identical:
\[
    \frac{p_{h,L,U}(y,t)}{p(y,t)}-1
    =
    \frac{Q_{h,L,U}(\eta,\lambda(t))}
    {Q_{\alpha,d}(\eta,\lambda(t))}-1 .
\]
Thus, if
\[
    \sup_{(\eta,\lambda)\in\mathcal D_{\mathcal W}}
    \left|\frac{Q_{h,L,U}(\eta,\lambda)}
    {Q_{\alpha,d}(\eta,\lambda)}-1\right|\le \epsilon,
\]
then
\[
    \sup_{(y,t)\in\mathcal W}
    \left|\frac{p_{h,L,U}(y,t)}{p(y,t)}-1\right|\le\epsilon .
\]
\end{proposition}
\begin{proof}
By Bernstein's theorem and \cref{eq:kwwfunrep},
\[
e^{-D_ft|\bm{k}|^{2\alpha}}
=\int_0^\infty e^{-s(D_ft)^{1/\alpha}|\bm{k}|^2}\rho_\alpha(s)\,ds .
\]
Substituting this identity into \cref{eq::p_intergral} and combining the
ordinary and fractional Gaussian factors gives
\[
D_ot+s(D_ft)^{1/\alpha}=(D_ft)^{1/\alpha}(s+\lambda(t)).
\]
The inverse Fourier transform of
$\exp[-(D_ft)^{1/\alpha}(s+\lambda)|\bm{k}|^2]$ is the heat kernel with variance
parameter $(D_ft)^{1/\alpha}(s+\lambda)$. Therefore
\[
p(y,t)=\int_0^\infty \rho_\alpha(s)
\big(4\pi(D_ft)^{1/\alpha}(s+\lambda)\big)^{-d/2}
\exp\!\left(-\frac{y^2}
{4(D_ft)^{1/\alpha}(s+\lambda)}\right)\,ds .
\]
Factoring out $(D_ft)^{-d/(2\alpha)}$ and using
$\eta=y/(D_ft)^{1/(2\alpha)}$ gives
\crefrange{eq::scaled_q}{eq::scaled_q_integral}. Replacing the integral
in $s$ by the finite trapezoidal rule gives \cref{eq::scaled_q_sog} and,
after undoing the scaling, the Gaussian sum
\cref{eq::Heat_explicit} with $L=-M_1$ and $U=M_2$. The relative-error
identity follows immediately because the positive scaling factor
$(D_ft)^{-d/(2\alpha)}$ is common to the exact and approximate profiles.
\end{proof}

\medskip\noindent\textbf{The relevant range of the kernel argument.}\;
The SOE approximates $e^{-x^\alpha}$, and in the solution its Fourier-side
argument is
$x=(D_ft)^{1/\alpha}|\bm{k}|^2$. Writing the solution radially and changing
variables from $|\bm{k}|$ to $x$ (the substitution used for $\mathcal{E}_\infty$
in \cref{sec::Error}),
\begin{equation}\label{eq::radial_x}
p(Y,t)\ \propto\ (D_ft)^{-\frac{d}{2\alpha}}\int_0^\infty x^{\,d/2-1}\,
\Lambda_d\!\big(\eta\sqrt{x}\big)\,e^{-x^\alpha}\,\mathrm{d}x,\qquad
\eta=\frac{|Y|}{(D_ft)^{1/(2\alpha)}},
\end{equation}
where
\[
\Lambda_d(z)=2^{(d-2)/2}\Gamma\!\left(\frac d2\right)
z^{-(d-2)/2}J_{(d-2)/2}(z)
\]
is the normalized radial kernel ($\Lambda_d(0)=1$). The solution therefore samples
$e^{-x^\alpha}$ only through the weight
$x^{d/2-1}\Lambda_d(\eta\sqrt{x})$. For the pure-fractional case $D_o=0$, a
physical window $|Y|\le R$, $t\in[t_{\min},t_{\max}]$ enters this scaled
description through the largest self-similar displacement
\[
\eta_{\max}=\frac{R}{(D_ft_{\min})^{1/(2\alpha)}}.
\]
The active scalar interval $x\in[x_{\min},x_{\max}]$ is chosen so that this
weight is retained on the region where its envelope exceeds an $\epsilon$-level
threshold. The two endpoints are governed by different data:
\begin{itemize}[leftmargin=1.5em,topsep=2pt,itemsep=1pt]
\item the upper end $x_{\max}$ is \emph{independent of $Y$ and $t$}: for $d>2$ the
radial envelope $x^{d/2-1}e^{-x^\alpha}$ peaks at
$x_\star=\big(\tfrac{d-2}{2\alpha}\big)^{1/\alpha}$, and $x_{\max}$ is the largest
$x$ with
$(\tfrac d2-1)\log\tfrac{x}{x_\star}-\big(x^\alpha-x_\star^\alpha\big)=\log\epsilon$.
For $d\le2$ the same upper-tail scale is obtained directly from
$e^{-x^\alpha}\lesssim\epsilon$. Thus
$x_{\max}\sim\max\!\big(x_\star,\,(\log\tfrac1\epsilon)^{1/\alpha}\big)$ for
$d>2$, and $x_{\max}\sim(\log\tfrac1\epsilon)^{1/\alpha}$ for $d\le2$;
\item the lower end $x_{\min}$ is the far-field, low-frequency cutoff set by
$\eta_{\max}$: as $\Lambda_d(\eta_{\max}\sqrt{x})$ departs from unity only once
$\eta_{\max}\sqrt{x}\gtrsim\sqrt{d}$, the window extends down to
$x_{\min}\sim x_\star/(1+\eta_{\max}^2)$.
\end{itemize}
Thus $x_{\max}$ (from $\alpha,d,\epsilon$) fixes the smallest node and hence
$M_1$, $x_{\min}$ (from $\eta_{\max}$) the largest node and hence $M_2$, and the
step $h$ controls the discretization error. When $D_o>0$, the Fourier-side factor
$e^{-D_ot|\bm{k}|^2}$ gives additional Gaussian damping; the comparison domain
in \cref{alg::SOE_selector} nevertheless retains the full dependence on
$\lambda(t)$.

\medskip\noindent\textbf{A domain-adapted selection procedure.}\;
When $D_o=0$, $\lambda=0$ and a fixed physical window $0\le y\le R$,
$t\in[t_{\min},t_{\max}]$ reduces to
\begin{equation}\label{eq::etamax_selector}
    0\le\eta\le\eta_{\max},\qquad
    \eta_{\max}=\frac{R}{(D_ft_{\min})^{1/(2\alpha)}},
\end{equation}
which formalizes Remark~\ref{rmk::selfsim}: the selection depends on $t_{\min}$,
through $\eta_{\max}$, but not on $t_{\max}$. When $D_o>0$, $\lambda(t)$ must be
retained over the whole interval in the comparison domain. Except in the
closed-form case $\alpha=1/2$, $D_o=0$, the profile $Q_{\alpha,d}$ is computed
from \cref{eq::scaled_q_integral}; the finite set
$\mathcal A\subset\mathcal D_{\mathcal W}$ in \cref{alg::SOE_selector}
approximates the supremum in Proposition~\ref{prop::Q_sog}.

\begin{algorithm}[ht]
\caption{Domain-adapted SOE parameter selection}\label{alg::SOE_selector}
\begin{algorithmic}[1]
\Require Fractional order $\alpha$; dimension $d$; tolerance $\epsilon$; spatial
radius $R$; time interval $[t_{\min},t_{\max}]$; coefficients $D_f,D_o$.
\State Define the comparison domain $\mathcal D$: if $D_o=0$, use
$\mathcal D=\{(\eta,0):0\le\eta\le\eta_{\max}\}$ with
\cref{eq::etamax_selector}; otherwise use
$\mathcal D=\{(\eta,\lambda(t)):t\in[t_{\min},t_{\max}],\,
0\le\eta\le R/(D_ft)^{1/(2\alpha)}\}$.
\State Construct an adaptive sample set $\mathcal A\subset\mathcal D$ containing
the endpoints and points uniform in $\log(1+\eta)$; subsequently augment it by
points where the relative error is largest.
\State Compute a high-accuracy reference $Q_{\rm ref}$ on $\mathcal A$ from
\cref{eq::scaled_q_integral} by high-accuracy quadrature in $u=\log s$; for
$\alpha=1/2$ and $D_o=0$, use the closed-form Cauchy profile instead.
\State For each step size $h$ under consideration, enlarge an initial index
interval $[L_{\rm w},U_{\rm w}]$ until the omitted positive tails of
\cref{eq::scaled_q_sog} are below the prescribed tail tolerances on
$\mathcal A$.
\State Evaluate
\[
E_{\mathcal A}(h,L,U)=
\max_{(\eta,\lambda)\in\mathcal A}
\left|\frac{Q_{h,L,U}(\eta,\lambda)}{Q_{\rm ref}(\eta,\lambda)}-1\right|
\]
in logarithmic form. Choose the largest $h$ for which
$E_{\mathcal A}(h,L_{\rm w},U_{\rm w})$ is below the prescribed discretization
tolerance.
\State With this $h$, remove terms from the left and right ends of the initial
interval until the shortest consecutive band $[L,U]$ satisfying
$E_{\mathcal A}(h,L,U)\le\epsilon$ is found. If no such band exists, decrease
$h$ and repeat.
\State Refine $\mathcal A$ near the observed maxima of the relative error and
repeat until further refinement leaves $E_{\mathcal A}$ unchanged to the
prescribed tolerance.
\Ensure Nodes $s_\ell=e^{\ell h}$, weights $w_\ell=h\,s_\ell\rho_\alpha(s_\ell)$,
and $\nexp=U-L+1$.
\end{algorithmic}
\end{algorithm}

The intermediate tolerances in \cref{alg::SOE_selector} enter only in the
preliminary determination of the step size and truncation interval. The final
SOG parameters are required to satisfy $E_{\mathcal A}\le\epsilon$ on the
adaptively refined set $\mathcal A$. For $\alpha=1/2$ and $D_o=0$,
$Q_{\rm ref}$ is the closed-form Cauchy profile; in the remaining cases, it is
computed from the scaled integral representation by high-accuracy quadrature.
This finite-set comparison is used for the domain-adapted parameter choice,
while the continuum a~priori estimate of \cref{sec::Error} provides the rigorous
error bound for the pure-fractional case $D_o=0$. In high dimensions, both
$Q_{\rm ref}$ and $Q_{h,L,U}$ are evaluated in logarithmic form, and the largest
term is factored out of the finite SOG sum to preserve numerical stability.

\medskip\noindent\textbf{Closed form at $\alpha=1/2$.}\;
The preceding construction does not rely on a closed form for
$\rho_\alpha$. When $\alpha=1/2$ and $D_o=0$, however, the one-sided stable
density is elementary,
$\rho_{1/2}(s)=(2\sqrt{\pi})^{-1}s^{-3/2}e^{-1/(4s)}$, and the discretization
error of the infinite trapezoidal rule can be written explicitly. These formulas
give a closed-form reference for the general parameter-selection criterion.
\begin{proposition}\label{prop::disc}
For $\alpha=\tfrac12$, let $\nu=(d+1)/2$, $\tau_m=2\pi m/h$, and
$\beta_Y=1+|Y|^2/(D_fT)^2$. For $D_o=0$ and every evaluation point
$Y=\bm{x}-\bm{x}_0^T$ and time $T$, the pointwise relative discretization error
of the infinite trapezoidal-rule approximation $p_h$ is
\begin{equation}\label{eq::disc_exact}
\frac{p_h(Y,T)-p(Y,T)}{p(Y,T)}
=\sum_{m\ne0}\left(\frac{4}{\beta_Y}\right)^{i\tau_m}
\frac{\Gamma(\nu+i\tau_m)}{\Gamma(\nu)} .
\end{equation}
Consequently,
\begin{equation}\label{eq::disc_bound}
\left|\frac{p_h(Y,T)-p(Y,T)}{p(Y,T)}\right|
\le
2\sum_{m\ge1}\frac{\big|\Gamma\!\big(\tfrac{d+1}{2}+i\,2\pi m/h\big)\big|}{\Gamma\!\big(\tfrac{d+1}{2}\big)},
\end{equation}
and this upper bound is independent of $Y$ and of $T$.
\end{proposition}
\begin{proof}
Using \cref{eq:kwwfunrep} with $\alpha=1/2$, the fractional heat kernel is the
positive mixture
\[
p(\bm{x},T)=\int_0^\infty\rho_{1/2}(s)(4\pi C)^{-d/2}
e^{-|Y|^2/4C}\,ds,\qquad C=s(D_fT)^2 .
\]
The approximation $p_{\mathrm{SOG}}$ is the trapezoidal rule for this integral
on the nodes $s_\ell=e^{h\ell}$. In the variable $\xi=\log s$ the integrand is
$A\,s^{-(d+1)/2}e^{-\beta_Y/4s}$ with $\beta_Y=1+|Y|^2/(D_fT)^2$; its Fourier
transform is $A\,(4/\beta_Y)^{(d+1)/2+i\tau}\Gamma(\tfrac{d+1}2+i\tau)$. Poisson
summation then gives \cref{eq::disc_exact} as the sum over nonzero Fourier
modes divided by the $\tau=0$ value. Since $|(4/\beta_Y)^{i\tau}|=1$, taking
absolute values gives \cref{eq::disc_bound}.
\end{proof}
In this closed-form case, a sufficient all-space choice of $h$ is the largest
value for which
\begin{equation}\label{eq::h_sharp}
2\sum_{m\ge1}\frac{\big|\Gamma\!\big(\tfrac{d+1}{2}+i\,2\pi m/h\big)\big|}{\Gamma\!\big(\tfrac{d+1}{2}\big)}
\le\epsilon
\end{equation}
holds. By Lemma~\ref{lemma:Gamma},
$|\Gamma(\nu+i\tau)|\sim\sqrt{2\pi}\,(\nu^2+\tau^2)^{(2\nu-1)/4}e^{-\pi|\tau|/2}$
decays exponentially in $|\tau|$. The leading-mode approximation, combined with
the Gaussian approximation of the Gamma ratio for $2\pi/h\ll (d+1)/2$, gives
the large-$d$ estimate
\begin{equation}\label{eq::h_sharp_asymp}
h\approx\frac{2\pi}{\sqrt{(d+1)\ln(2/\epsilon)}}.
\end{equation}
The integrand peaks at $s_\star=\beta_Y/(2(d+1))$ and decays as
$\exp(-\tfrac{d+1}2\psi(v))$, where
$\psi(v)=v+e^{-v}-1$ and $v=\log(s/s_\star)$. Equating this decay factor to
$\epsilon$ gives the retained band
\begin{equation}\label{eq::band}
s_\ell=e^{h\ell}\in\Big[\frac{e^{v_{\mathrm{lo}}}}{2(d+1)},\
\frac{(1+\eta_{\max}^2)\,e^{v_{\mathrm{hi}}}}{2(d+1)}\Big],\qquad
\eta_{\max}=\frac{R}{(D_fT_{\min})^{1/(2\alpha)}},
\end{equation}
where the endpoints $v_{\mathrm{lo}}<0$ and $v_{\mathrm{hi}}>0$ solve
\begin{equation}
    \label{eq::vlo_vhi}
    \psi(v_{\mathrm{lo}})=\psi(v_{\mathrm{hi}})=q_d:=\frac{2\log(1/\epsilon)}{d+1}.
\end{equation}
The corresponding term count is
\begin{equation}\label{eq::nexp}
\nexp=M_1+M_2+1=\frac{\ln(1+\eta_{\max}^2)+(v_{\mathrm{hi}}-v_{\mathrm{lo}})}{h}.
\end{equation}
\begin{remark}\label{rmk::nexp_d}
Since $h\propto(d+1)^{-1/2}$ in the large-$d$ estimate
\cref{eq::h_sharp_asymp}, $\nexp$ is non-monotonic in $d$. At small $d$ the
threshold $q_d$ in \cref{eq::vlo_vhi} is large, providing that the band width dominates
with
\begin{equation}
    \label{eq::n_exp_small_d}
    \frac{v_{\mathrm{hi}}-v_{\mathrm{lo}}}{h}\approx\frac{q_d+1+\log q_d}{h}
    \sim O\!\big((d+1)^{-1/2}\big),
\end{equation}
and thus $\nexp$ decreases. It reaches a minimum near $d\approx29$; once $q_d$ is
small, the displacement term $\ln(1+\eta_{\max}^2)/h\propto\sqrt{d}$ becomes
dominant and $\nexp$ grows again.
\end{remark}

The complete FFPE solver is summarized in \cref{alg::SOG_FFPE}.

\begin{algorithm}[ht]
\caption{The SOG fast FFPE solver}\label{alg::SOG_FFPE}
\begin{algorithmic}[1]
\Require Drift vector $\bm{b}$; ordinary diffusion coefficient $D_o$; fractional diffusion coefficient $D_f$; initial point $\bm{x}_0$; terminal time $T$; $d$-dimensional tensor evaluation grid $\bm{X}=\otimes_{j=1}^{d}\{X_{j}\}$, where $X_j=\{x_{j}^{1},\cdots,x_{j}^{n_j}\}$ is the grid in the $j$-th coordinate; accuracy tolerance $\epsilon$.

\State According to $\epsilon$, select the SOG parameters from
\cref{alg::SOE_selector}; in the pure-fractional case, a formal a~priori
guarantee can be obtained from the conservative construction of
\cref{sec::Error}.

\State For each SOG term $\ell=-M_1,\cdots,M_2$, assemble the
one-dimensional Gaussian factors of $p_\ell(\bm{x},T)$ in
\cref{eq::Heat_explicit} on the coordinate sets $X_j$. Each term can be
assembled on its own with $O(\sum_{j=1}^{d}n_j)$ auxiliary storage, or all
factors can be stored with $O(\nexp\sum_{j=1}^{d}n_j)$ storage.

\Ensure The separated approximation $p_{\mathrm{SOG}}$ in
\cref{eq::p_approx}, stored through its one-dimensional Gaussian factors with
$O(\nexp\sum_{j=1}^{d}n_j)$ storage; dense grid values can be generated from this
representation when requested.
\end{algorithmic}
\end{algorithm}

\section{Numerical results}
\label{sec::Num}

In this section we assess the fast SOG solver of \cref{alg::SOG_FFPE}
and study its accuracy, robustness, and computational cost. All experiments are
carried out in MATLAB R2025b on a laptop with an Intel Core Ultra 7~255H CPU and
64\,GB of memory, with a serial implementation. As an
exact reference we use the only nontrivial case for which the
fundamental solution is known in closed form in every dimension: the pure
fractional case $D_o=0$, $\alpha=1/2$, for which \cref{eq::int_expression}
evaluates to the $d$-dimensional Cauchy distribution \cite{ye2026fast}
\begin{equation}
    \label{eq::cauchy_ref}
    p(\bm{x},T)=\frac{\Gamma\!\left(\frac{d+1}{2}\right)}{\pi^{\frac{d+1}{2}}}
    \frac{D_fT}{\big[(D_fT)^2+y^2\big]^{\frac{d+1}{2}}},
    \qquad y=|\bm{x}-\bm{x}_0^T|.
\end{equation}
This reference is exact to machine precision for arbitrary $d$, and is therefore
well suited to assessing the solver in the high-dimensional and small-time
regimes that are otherwise the most difficult to compare against closed-form
references.

All reported term counts are the consecutive bands returned by
\cref{alg::SOE_selector}. A run is fully specified by
$(\alpha,d,D_o,D_f,\epsilon)$, the physical window in $y$ and $t$, and the
resulting parameters $(h,L,U)$. The SOG weights are then
$w_\ell=h e^{h\ell}\rho_\alpha(e^{h\ell})$, and all high-dimensional sums and
references are evaluated through logarithms using \cref{eq::logsumexp} and
log-Gamma functions. For $\alpha=1/2$ errors are measured against the exact
Cauchy density. For $\alpha\ne1/2$ the reported errors are the
solution-level self-convergence estimates described in \cref{sec::Num_validation}.

\subsection{High-order accuracy and convergence}
\label{sec::Num_conv}

\begin{figure}[!t]
    \centering
    \includegraphics[width=\textwidth]{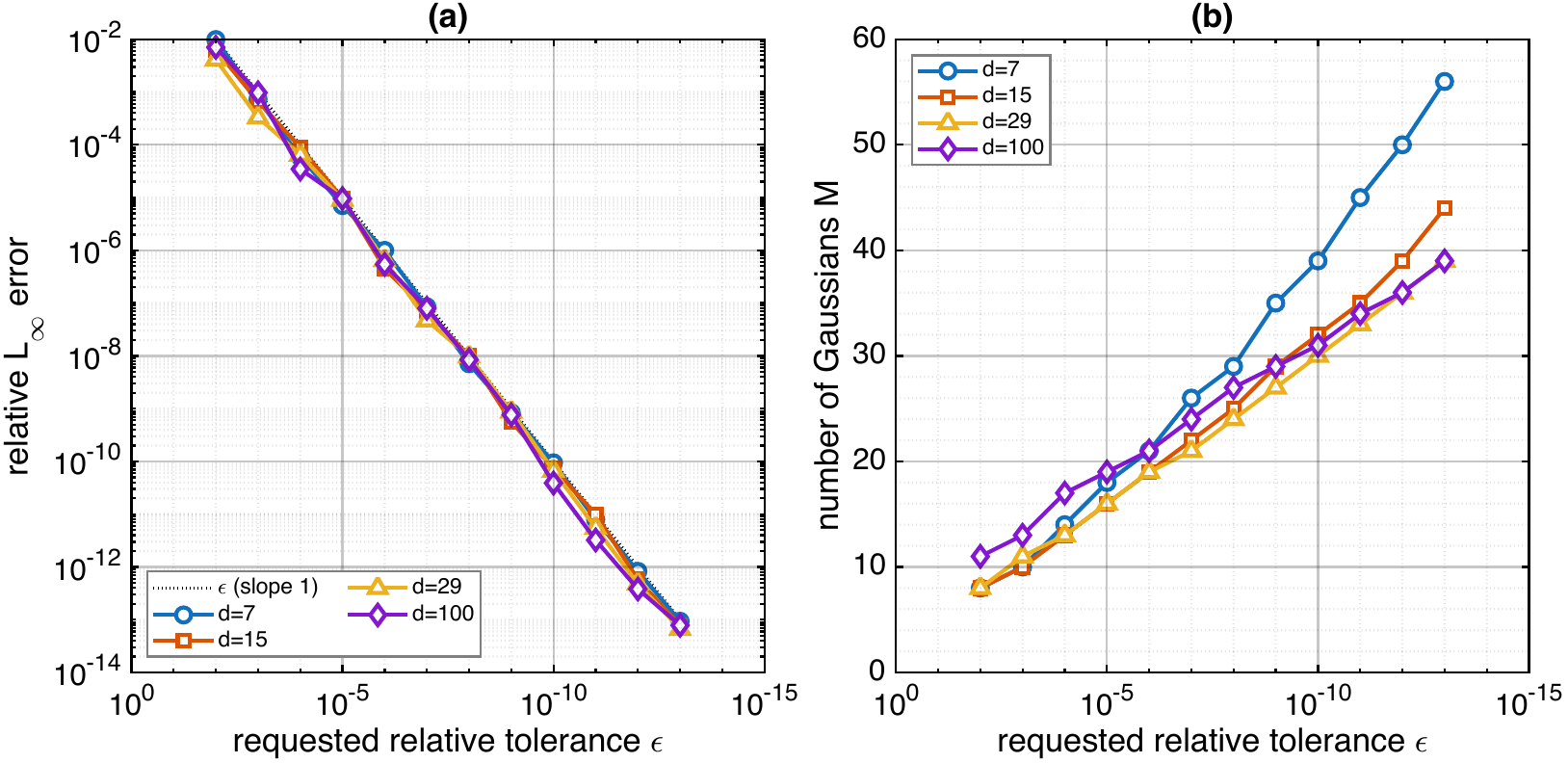}
    \caption{High-order convergence of the SOG solver against the exact Cauchy
    solution~\eqref{eq::cauchy_ref} ($D_o=0$, $\alpha=1/2$, $T=1$).
    (a)~relative $L_\infty$ error versus the prescribed tolerance
    $\epsilon$; the dotted line is the slope-one reference. (b)~the number of
    Gaussian terms $\nexp$ grows only logarithmically in $1/\epsilon$.}
    \label{fig::convergence}
\end{figure}

We first verify that the SOG solver attains a prescribed accuracy in
moderate and high dimension. Fixing $D_f=1$, $D_o=0$, $\alpha=1/2$, and $T=1$,
we prescribe a relative accuracy $\epsilon$, select the SOG parameters by
\cref{alg::SOE_selector} (with the step $h$ initialized by
\cref{eq::h_sharp} and the band by \cref{eq::band,eq::nexp}),
and measure the relative $L_\infty$ error of $p_{\text{SOG}}$ against the exact
Cauchy density~\eqref{eq::cauchy_ref} over $y\in[0,2]$. The dimensions in
\cref{fig::convergence} are $d=7,15,29,100$. Panel~(a) shows that the
measured error tracks the prescribed tolerance with slope one across eleven
orders of magnitude, reaching $7.8\times10^{-14}$ at $d=100$ when
$\epsilon=10^{-13}$. The accuracy is therefore limited by the prescribed
tolerance and, ultimately, by the arithmetic of the double-precision evaluation. Two further observations
confirm the solution-level parameter selection of \cref{sec::params}:
\begin{itemize}[leftmargin=1.5em]
\item The number of Gaussians $\nexp$ grows only logarithmically in $1/\epsilon$
(\cref{fig::convergence}(b)). Across all dimensions and tolerances shown,
$\nexp$ lies between $8$ and $56$; at $\epsilon=10^{-13}$ the term counts are
$56,44,39,39$ for $d=7,15,29,100$, respectively. The dimension enters through
both the step $h\propto(d+1)^{-1/2}$ (asymptotically,
\cref{eq::h_sharp_asymp}) and the truncation band
\cref{eq::band,eq::nexp}. The decrease over this range is the
low-to-moderate-dimensional branch of the non-monotonic dependence described in
Remark~\ref{rmk::nexp_d}; in very high dimension the displacement-band term
$\ln(1+\eta_{\max}^2)/h\propto\sqrt d$ dominates and $\nexp$ grows again
(\cref{sec::Num_highd}).
\item Replacing the evaluation of $\rho_\alpha$ by the exact closed form of the
one-sided stable (L\'evy) density $\rho_{1/2}(s)=\tfrac{1}{2\sqrt{\pi}} s^{-3/2}e^{-1/(4s)}$ reproduces the errors of \cref{fig::convergence} to
every displayed digit. The density evaluation therefore does not limit
the accuracy; the residual error is due entirely to the SOE truncation.
\end{itemize}
For $\alpha=1/2$, Proposition~\ref{prop::disc} controls the relative
discretization error independently of the evaluation point and of $T$, and
\cref{alg::SOE_selector} enforces $E_{\mathcal A}\le\epsilon$ for the finite
retained band on the prescribed spatial window; by
Proposition~\ref{prop::Q_sog}, this scaled relative error equals that of the
physical Green's function, so the guarantee transfers directly to
$p_{\text{SOG}}$ and the measured error tracks $\epsilon$ uniformly in $d$.

\begin{remark}[Range of fractional orders]\label{rmk::alpha_range}
The experiments reported here use $\alpha\in[0.1,0.9]$. As $\alpha\to0^+$ or
$\alpha\to1^-$ the bound on $I_\alpha$ in Theorem~\ref{thm::I_alpha_bound}
diverges, so the trapezoidal step $h$ shrinks and the term count $\nexp$ grows
rapidly; this is intrinsic, since the representing density becomes singular at
both endpoints, degenerating to a Dirac measure at $z=1$ as $\alpha\to1^-$ and
to a nonregular limiting object as $\alpha\to0^+$. Independently, direct
evaluation of $\rho_\alpha$ becomes ill-conditioned near the endpoints: in our
implementation, the relative error of the Laplace identity
$\int_0^\infty\rho_\alpha(t)e^{-xt}\mathrm{d}t=e^{-x^\alpha}$ exceeds $10^{-3}$
for $\alpha\lesssim0.05$ and $\alpha\gtrsim0.99$, while it is at the level of
machine precision throughout $\alpha\in[0.1,0.95]$. Accurate computation in the
immediate vicinity of the endpoints requires a dedicated evaluator for
$\rho_\alpha$ (for instance, numerical steepest descent on the contour of
Appendix~\ref{sec::rho_estimate}) and is left to future work.
\end{remark}

\subsection{Self-convergence study}
\label{sec::Num_validation}

The exact Cauchy reference~\eqref{eq::cauchy_ref} is available only for
$\alpha=1/2$. For other fractional orders we assess the assembled solution by
\emph{self-convergence}: the chosen finite Gaussian sum is compared, on the
scaled radial interval, with a refined reference at half the step and a wider
band. Since the inverse Fourier transform of each Gaussian is exact, this probes
the only numerical approximation in the method, the SOE approximation of
$e^{-x^\alpha}$ on the active solution window. These entries should therefore be
interpreted as finite-domain self-convergence estimates; the independent
exact-reference checks are the $\alpha=1/2$ row and the Cauchy tests in the
other tables.

\begin{table}[!th]
    \centering
    \begin{tabular*}{\textwidth}{@{\extracolsep{\fill}}c|rrrr@{}}
        \hline
        $\alpha$ & $\eta_{\max}$ & $\nexp$ & $h$ & error\\
        \hline
        0.1 & $5.96\,\textrm{e}{2}$ &  2492 & $3.43\,\textrm{e}{-2}$ & $4.5\,\textrm{e}{-12}$\\
        0.3 & $1.34\,\textrm{e}{1}$ &   631 & $2.67\,\textrm{e}{-2}$ & $1.5\,\textrm{e}{-11}$\\
        0.5 & $6.25\,\textrm{e}{0}$ &   126 & $3.43\,\textrm{e}{-2}$ & $9.1\,\textrm{e}{-13}$\\
        0.7 & $4.51\,\textrm{e}{0}$ &    25 & $2.08\,\textrm{e}{-2}$ & $1.2\,\textrm{e}{-11}$\\
        0.9 & $3.77\,\textrm{e}{0}$ &    21 & $9.89\,\textrm{e}{-3}$ & $6.5\,\textrm{e}{-12}$\\
        \hline
    \end{tabular*}
    \caption{Self-convergence across fractional orders in $d=1000$. Each row
    uses \cref{alg::SOE_selector} at target relative tolerance
    $\epsilon=10^{-10}$ for $D_o=0$, $D_f=8$, $t=0.04$, and $y\in[0,2]$.
    For $\alpha\ne1/2$, the error is the finite-domain self-convergence estimate;
    for $\alpha=1/2$, it is the true error against the Cauchy
    density~\eqref{eq::cauchy_ref}.}
    \label{tab::selfconv}
\end{table}

\Cref{tab::selfconv} reports the result at $d=1000$ for
$\alpha=0.1,0.3,0.5,0.7,0.9$, with $D_o=0$, $D_f=8$, $t=0.04$, and $y\in[0,2]$,
at prescribed tolerance $\epsilon=10^{-10}$. It also lists the scaled endpoint
$\eta_{\max}=R/(D_ft)^{1/(2\alpha)}$, since this -- not $t$ alone -- sets the
active range of Gaussian scales. The largest term count occurs at $\alpha=0.1$,
where the fixed window maps to the largest scaled radius; the $\alpha=0.3$ row,
in which a saddle-point approximation of $\rho_\alpha$ is blended smoothly with
the direct evaluation at small $s$, indicates that the density evaluation is not
the limiting factor in this test. The $\alpha=1/2$ row, where the true error
against \cref{eq::cauchy_ref} is available, also meets the tolerance and anchors
the self-convergence estimates at the other orders.

\subsection{Robustness across time and dimension}
\label{sec::Num_robust}

We now fix a relative tolerance $\epsilon=10^{-12}$ and, for each $(t,d)$, select
the SOG parameters from \cref{alg::SOE_selector}
(with the $\alpha=1/2$ step from \cref{eq::h_sharp} and band from
\cref{eq::band,eq::nexp}). We evaluate
$p_{\text{SOG}}$ for $D_o=0$, $D_f=8$, $\alpha=1/2$ on the higher-dimensional part
of the $t\times d$ grid used in \cite{ye2026fast}, reporting the maximum relative
error over $y\in[0,2]$.

\begin{table}[!th]
    \centering
    \resizebox{\textwidth}{!}{%
    \begin{tabular}{r|ccccccc}
        \hline
        $d$ & 5 & 9 & 13 & 17 & 21 & 25 & 29\\
        \hline
        SOG (tested $t$ grid) & 1.0\,\textrm{e}{-12} & 1.0\,\textrm{e}{-12}
            & 1.0\,\textrm{e}{-12} & 1.0\,\textrm{e}{-12} & 1.0\,\textrm{e}{-12}
            & 1.0\,\textrm{e}{-12} & 1.0\,\textrm{e}{-12}\\
        \cite{ye2026fast}, $t{=}0.004$ & 2.2\,\textrm{e}{-9}
            & 2.2\,\textrm{e}{-6} & 3.2\,\textrm{e}{-3} & 1.7\,\textrm{e}{+0}
            & 5.4\,\textrm{e}{+2} & 2.3\,\textrm{e}{+6} & 3.1\,\textrm{e}{+9}\\
        \cite{ye2026fast}, $t{=}0.2$ & 7.7\,\textrm{e}{-16}
            & 9.6\,\textrm{e}{-16} & 1.7\,\textrm{e}{-15} & 1.9\,\textrm{e}{-15}
            & 2.8\,\textrm{e}{-15} & 5.5\,\textrm{e}{-15} & 7.1\,\textrm{e}{-15}\\
        \hline
    \end{tabular}}
    \caption{Maximum relative error over $y\in[0,2]$ for the pure-fractional
    case $D_o=0$, $D_f=8$, $\alpha=1/2$. The SOG row uses a separate
    quadrature for each dimension at relative tolerance $\epsilon=10^{-12}$; a
    single approximation, with its terms fixed at the smallest time
    $t_{\min}=0.004$, covers the tested time grid and uses
    $\nexp=71$--$82$ terms over this range.
    On this tested grid, the integral solver of~\cite{ye2026fast} is accurate at
    moderate time but loses accuracy rapidly at the smallest time as $d$
    increases.}
    \label{tab::headtohead}
\end{table}

\begin{figure}[!t]
    \centering
    \includegraphics[width=0.52\textwidth]{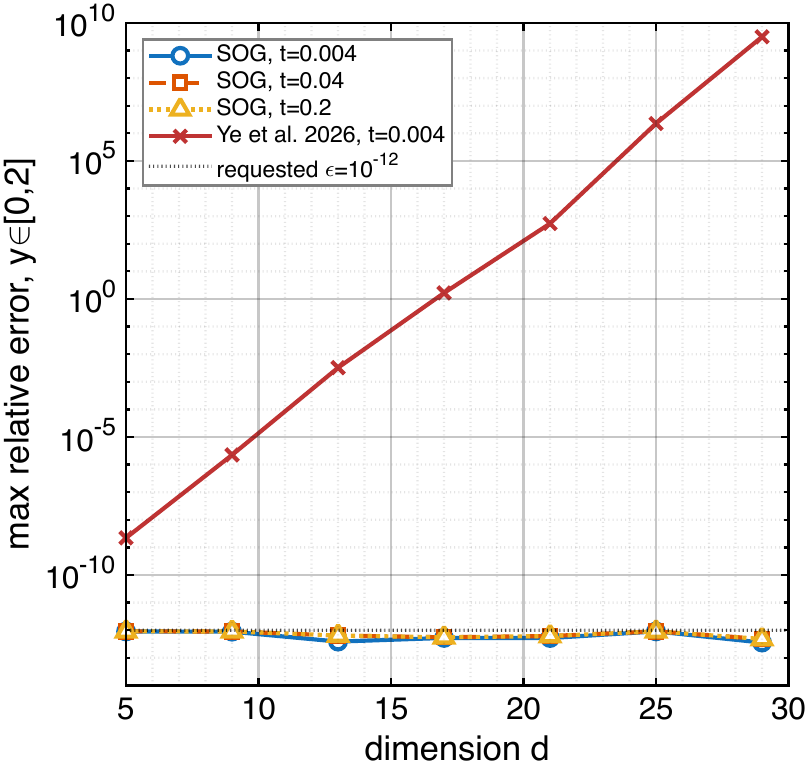}
    \caption{Maximum relative error versus dimension for the pure-fractional case
    ($D_o=0$, $D_f=8$, $\alpha=1/2$). The three SOG curves for $t=0.004,0.04,0.2$
    coincide and stay at the prescribed tolerance $\epsilon=10^{-12}$ (dotted),
    confirming that the approximation determined by $\eta_{\max}$ is reused across the time
    grid and remains uniform in $d$. For
    comparison, the integral-quadrature solver of \cite{ye2026fast} at
    $t=0.004$ (red) loses accuracy rapidly as $d$ grows, reaching
    $\mathcal{O}(10^{9})$ at $d=29$ on this test.}
    \label{fig::robustness}
\end{figure}

Two features stand out. First, the time dependence is handled by self-similar
scaling (Proposition~\ref{prop::disc} and Remark~\ref{rmk::selfsim}): a single
approximation sized at $t_{\min}=0.004$, using $\nexp=71$--$82$ terms over
$d=5$--$29$, holds relative error about $10^{-12}$ across the whole time grid,
with no degradation as $t\to0^+$ (the regime in which the solution concentrates
toward the Dirac measure). Second, the error is \emph{uniform in dimension},
staying at the prescribed $10^{-12}$ throughout. \Cref{tab::headtohead}
contrasts this with the integral solver of~\cite{ye2026fast}: at small time
($t=0.004$) its relative error grows rapidly over the tested dimensions,
exceeding unity for $d\ge17$ and reaching $\mathcal{O}(10^{9})$ at $d=29$,
whereas at moderate time ($t=0.2$) it attains machine precision. The two
approaches are thus
\emph{complementary}: the integral solver of~\cite{ye2026fast} is effective for
high-precision evaluation in low to moderate dimension, while the SOG holds the
prescribed accuracy uniformly across the tested $(t,d)$ range -- including the
small-$t$, high-$d$ corner where the integral solver loses accuracy, as shown in
\cref{fig::robustness} -- and scales to much higher dimensions
(\cref{sec::Num_highd}).

\subsection{Scaling to very high dimension}
\label{sec::Num_highd}

Unlike the oscillatory radial integrand $r^{d/2}J_{(d-2)/2}(yr)$ that limits
\cite{ye2026fast} to moderate $d$, the SOG solution is a \emph{positive} sum of
separable Gaussians, so its logarithm can be evaluated stably. Setting
$a_\ell(\bm{x})=\log w_\ell-\tfrac{d}{2}\log(4\pi C_\ell^{T})
-|\bm{x}-\bm{x}_0^T|^2/(4C_\ell^{T})$ and $a_\star=\max_\ell a_\ell(\bm{x})$,
\begin{equation}
\label{eq::logsumexp}
\log p_{\text{SOG}}(\bm{x},T)=a_\star+\log\sum_{\ell}
\exp\!\big(a_\ell(\bm{x})-a_\star\big),
\end{equation}
so that the per-term magnitudes $(4\pi C_\ell^{T})^{-d/2}$, which overflow or
underflow for large $d$, never appear explicitly. Each evaluation requires
$O(\nexp)$ operations for the radial profile.

\begin{table}[!t]
    \centering
    \begin{tabular*}{\textwidth}{@{\extracolsep{\fill}}r|cccc@{}}
        \hline
        $d$ & $100$ & $1{,}000$ & $10{,}000$ & $100{,}000$\\
        \hline
        terms $\nexp$    & $36$ & $63$ & $156$ & $451$\\
        relative error      & $9.9\,\textrm{e}{-13}$ & $1.6\,\textrm{e}{-12}$ & $7.3\,\textrm{e}{-12}$ & $5.8\,\textrm{e}{-11}$\\
        time per point (\textmu s) & $0.33$ & $0.48$ & $0.96$ & $4.2$\\
        \hline
    \end{tabular*}
    \caption{Very-high-dimensional evaluation of the SOG fundamental solution
    ($\alpha=1/2$, $D_o=0$, $D_f=1$, $T=1$), with relative error measured against
    the exact $d$-dimensional Cauchy solution~\eqref{eq::cauchy_ref} computed
    through its logarithm. Each dimension uses its own approximation, with the number of terms fixed by \cref{eq::nexp} and the sum evaluated through its logarithm via \cref{eq::logsumexp}; the cost per evaluation point is $O(\nexp)$.}
    \label{tab::highd}
\end{table}

These dimensions exceed those reached in the radial-quadrature experiments
of~\cite{ye2026fast}. We again use the exact Cauchy
solution~\eqref{eq::cauchy_ref}: although its magnitude over- or underflows, its
\emph{logarithm} is computable to full relative precision in any dimension via
the log-Gamma function. \Cref{tab::highd} reports the relative error and
evaluation time per point for $d$ up to $10^{5}$. The SOG attains ten-digit
relative accuracy at $d=10^{5}$ in about $4$\,\textmu s per point, with the term
count $\nexp$ rising only from
$36$ at $d=100$ to $451$ at $d=10^5$ (the slow large-$d$ growth is analyzed in
Remark~\ref{rmk::nexp_d}). These tests
give deterministic, high-accuracy evaluations at dimensions beyond the reach of
radial-quadrature methods.

\subsection{General initial data: tensor-product representations}
\label{sec::Num_sog_ic}

A key advantage of the SOG representation is that it provides a reusable building
block that maps separated (tensor-product) data to separated data. The
SOG approximation $p_{\text{SOG}}$ is itself a rank-$\nexp$ separated function: by
\cref{eq::Heat_explicit} each term $p_\ell$ is a tensor product
$\prod_{j=1}^{d} g_\ell(x_j)$ of one-dimensional Gaussians. Consequently, for any
initial datum written as a linear combination of tensor products,
\begin{equation}
    \label{eq::tensor_ic}
    p(\bm{x},0)=\sum_{r=1}^{R}\prod_{j=1}^{d}\phi_{r,j}(x_j),
\end{equation}
linearity and the separability of the heat kernel give the solution in closed
form as
\begin{equation}
    \label{eq::tensor_sol}
    p(\bm{x},T)=\sum_{r=1}^{R}\sum_{\ell=-M_1}^{M_2}w_\ell\prod_{j=1}^{d}
    \big(g_\ell * \phi_{r,j}\big)(x_j-b_jT),
\end{equation}
which is again a sum of tensor products. The dimension $d$ enters only through the \emph{one-dimensional} convolutions
$g_\ell*\phi_{r,j}$, and the rank grows from $R$ to $R\nexp$, recompressible by
standard tensor-rank truncation. Such separated, low-rank formats -- the
canonical, Tucker, and tensor-train decompositions
\cite{beylkin2002numerical,beylkin2005algorithms,hackbusch2012tensor,kolda2009tensor,oseledets2011tt}
and the related tensor networks
\cite{biamonte2017nutshell,orus2014practical,schollwock2011dmrg} -- provide a natural
high-dimensional function class for this solver. Tensor neural
networks
\cite{wang2023tnnsurvey,wang2024posteriori,wang2024multieigenpairs,wang2024tensor}
can be used to construct separated approximations for more general initial data.

\begin{figure}[!t]
    \centering
    \begin{minipage}[b]{0.48\textwidth}
        \centering
        \includegraphics[width=\linewidth]{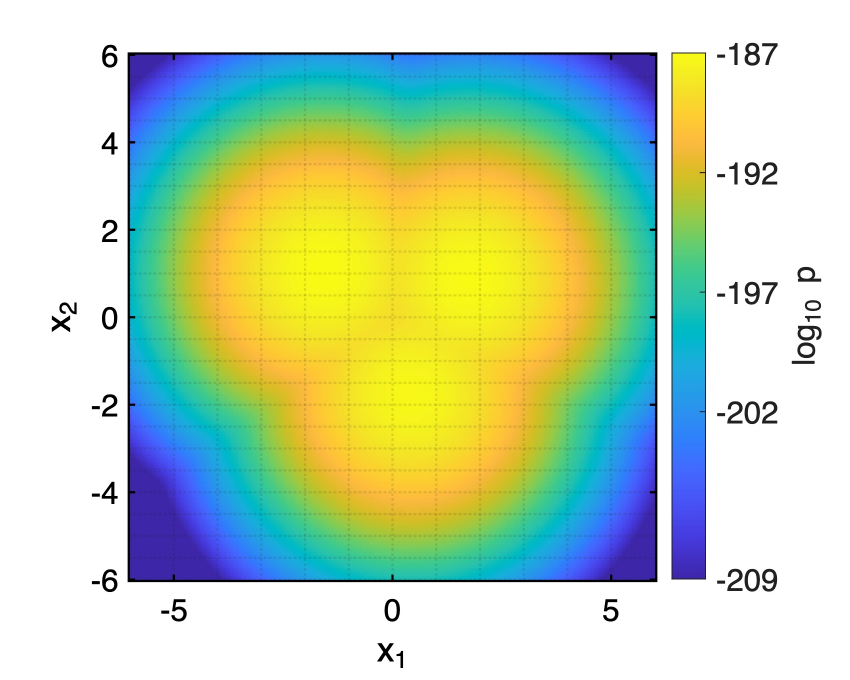}\\[2pt](a)
    \end{minipage}\hfill
    \begin{minipage}[b]{0.40\textwidth}
        \centering
        \includegraphics[width=\linewidth]{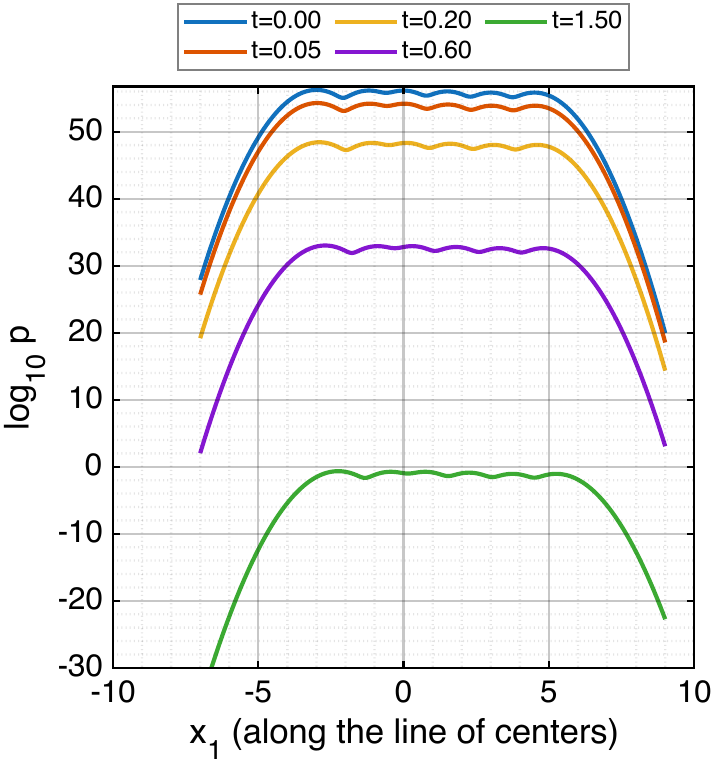}\\[2pt](b)
    \end{minipage}
    \caption{Sum-of-Gaussians solutions in $d=1000$ ($\alpha=1/2$, $D_o=0$,
    $D_f=1$), assembled in closed form from \cref{eq::sog_ic} and evaluated through
    their logarithm \cref{eq::logsumexp}. (a)~A two-dimensional slice at $t=0.4$
    of a three-Gaussian initial condition with drift $\bm{b}=(0.6,-0.4,0,\dots,0)$
    in the $x_1$--$x_2$ plane, shown as $\log_{10}p$ (top $22$ decades); the
    structure is non-radial and heavy-tailed. (b)~Evolution of a six-Gaussian
    initial condition along the line of centers, shown as $\log_{10}p$; the peak
    density exceeds $10^{50}$ and decays over more than eighty orders of magnitude
    as the packet spreads with the characteristic heavy fractional tails.}
    \label{fig::highdim}
\end{figure}

Gaussian initial data are the canonical example: the one-dimensional
convolutions remain Gaussian and are available analytically. Since each
$p_\ell$ in \cref{eq::Heat_explicit} has covariance $2C_\ell^{T}I$, the solution
for the initial condition
\[
p(\bm{x},0)=\sum_{j=1}^{N_g}a_j\,
\mathcal{N}(\bm{x};\bm{c}_j,\sigma_j^2 I)
\]
is
\begin{equation}
    \label{eq::sog_ic}
    p(\bm{x},T)=\sum_{j=1}^{N_g}\sum_{\ell=-M_1}^{M_2}a_j\,w_\ell\,
    \mathcal{N}\!\big(\bm{x};\,\bm{c}_j+\bm{b}T,\,(\sigma_j^2+2C_\ell^{T})I\big),
\end{equation}
a sum of $N_g\nexp$ Gaussians, each a tensor product across the $d$ coordinates,
stored and manipulated in the same factored $O(\nexp dN)$ representation as the
fundamental solution.

We test \cref{eq::sog_ic} directly in high dimension. For a single
Gaussian source, the value at the advected center has the independent
non-oscillatory reference
\[
\frac{\Omega_d}{(2\pi)^d}\int_0^\infty
r^{d-1}\exp\!\left(-\frac{\sigma^2r^2}{2}-D_fT r^{2\alpha}\right)\mathrm{d}r,
\]
which is evaluated after a saddle-point change of variables and does not involve the oscillatory
Bessel integral. With $\alpha=1/2$, $\sigma=0.5$, $D_f=1$, and a broad
$\nexp=3251$ term SOE that preserves $\sum_\ell w_\ell=1$ to the displayed digits, the
relative errors at $T=0.1$ are $1.1\times10^{-13}$, $2.7\times10^{-12}$, and
$6.2\times10^{-11}$ for $d=10^3,10^4,10^5$, respectively; at $T=1$ they are
$3.4\times10^{-13}$, $9.1\times10^{-13}$, and $9.8\times10^{-11}$.

\Cref{fig::highdim}(a) shows a two-dimensional slice of the $d=1000$
solution for a non-radial, three-Gaussian initial condition under drift and pure
fractional diffusion, computed from the separated representation directly rather
than by reducing the data to a radial profile.

\begin{figure}[!t]
    \centering
    \includegraphics[width=0.52\textwidth]{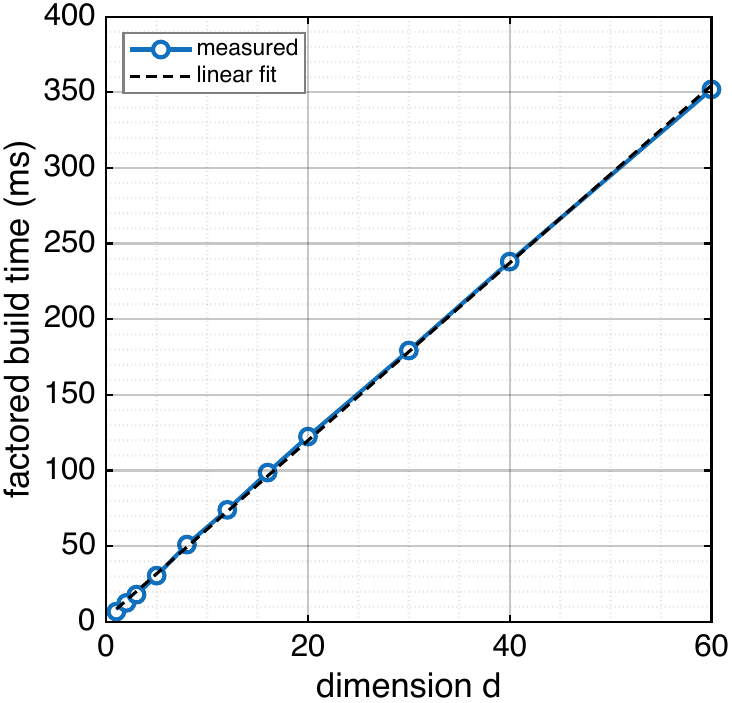}
    \caption{Cost of assembling the factored sum-of-Gaussians solution
    \cref{eq::sog_ic} versus dimension, for $N_g=8$ sources and
    $N=64$ points per dimension. The assembly time grows linearly in $d$, and the
    storage is $d$ factor matrices of size $N\times P$ with $P=N_g\nexp$.}
    \label{fig::scaling}
\end{figure}

Finally, \cref{fig::scaling} confirms the linear-in-$d$ cost. Holding the
SOG approximation fixed at $\nexp=1012$ terms ($P=N_g\nexp=8096$ Gaussians for $N_g=8$
sources) and varying only the dimension, the factored assembly time is
indistinguishable from a straight line through the origin, reaching $0.38$\,s at
$d=60$. In $d=10$ the same $P=8096$ factored Gaussians are assembled in $0.07$\,s,
and a factored evaluation agrees with the explicit double sum over $j$ and $\ell$
to $8\times10^{-16}$ -- whereas the corresponding dense tensor has
$N^{10}\approx10^{18}$ entries and is not a feasible object to store.

Combining this building block with the logarithmic evaluation of
\cref{sec::Num_highd} lets us evolve sum-of-Gaussians initial data in very high
dimension; \Cref{fig::highdim}(b) shows a six-Gaussian initial condition evolved
in $d=1000$ with drift along the line of centers. Spanning more than eighty orders
of magnitude, the solution is representable only through its
logarithm~\eqref{eq::logsumexp}. The mass $\int_{\mathbb{R}^d}p\,\mathrm{d}\bm{x}
=\big(\sum_j a_j\big)\big(\sum_\ell w_\ell\big)$ is conserved to $1.000000$ at
every time, in any dimension, and evaluating one 600-point time slice costs about
$0.09$\,s at $d=1000$.

\section{Conclusion}
\label{sec::conclusion}

We have developed a sum-of-Gaussians solver for the high-dimensional fractional
Fokker--Planck equation with high-order accuracy and error control. The
construction uses the complete monotonicity of the fractional symbol
$e^{-D_ft|\bm{k}|^{2\alpha}}$, which represents it as a continuous
superposition of Gaussians. Discretizing that superposition gives a finite sum
of ordinary heat flows that act independently in each coordinate, with storage
and assembly cost linear in the dimension and quadrature parameters fixed
a~priori for any target accuracy. In our experiments the solver attains more than
ten digits of accuracy, with $\nexp$ growing only logarithmically as the
tolerance tightens and the accuracy uniform across the time windows tested;
evaluating the positive Gaussian form through its logarithm carries the
computation to dimension $10^{5}$.

Because the method approximates the fundamental solution by a separated sum of Gaussians, it
propagates any tensor-product initial datum in closed form, so the only
additional ingredient needed for more general data is a separated, low-rank
representation of it. Sparse-grid and tensor neural network approximation
frameworks are compatible with this requirement and can be combined directly
with the present solver. When a high-dimensional Fourier symbol or interaction
kernel admits an accurate Gaussian-sum expansion, the expansion expresses
nonseparable terms as sums of separable Gaussian contributions. The FFPE results
presented here therefore support the use of SOG approximations as building
blocks for separable representations of kernels and solution operators in
high-dimensional PDEs.

The approach is also relevant beyond the particular FFPE fundamental solution
studied here. For a linear constant-coefficient evolution equation in high
dimension, if the Fourier-space propagator admits an accurate Gaussian-sum
approximation, the same construction expresses the solution operator as a sum of
separable Gaussian evolution operators. For many nonlinear evolution equations,
applying an unconditionally energy-stable scalar auxiliary variable (SAV)
temporal discretization reduces each time step to a linear problem with known
source terms \cite{ShenXuYang2018}. When this linear problem has a
constant-coefficient solution operator that admits an accurate and efficient
Gaussian-sum representation, the present solver can serve as a building block for
high-dimensional nonlinear evolution PDEs.

\section*{Acknowledgments}
The Flatiron Institute is a division of the Simons Foundation. The work of
Q.~Zhou was supported by the National
Natural Science Foundation of China (Grant No.~125B2023).

\appendix
\renewcommand{\theHsection}{appendix.\Alph{section}}
\renewcommand{\theHsubsection}{appendix.\Alph{section}.\arabic{subsection}}
\renewcommand{\theHfigure}{appendix.\Alph{section}.\arabic{figure}}
\renewcommand{\theHtable}{appendix.\Alph{section}.\arabic{table}}
\renewcommand{\theHequation}{appendix.\Alph{section}.\arabic{equation}}
\providecommand{\theHtheorem}{}
\renewcommand{\theHtheorem}{appendix.\Alph{section}.\arabic{theorem}}

\section{A bound for \texorpdfstring{$|\rho_\alpha(z)|$}{|rho\_alpha(z)|} on \texorpdfstring{\mbox{$|z|\le 1$}}{|z| <= 1}}
\label{sec::rho_estimate}
In this section, we provide some estimates on
$|\rho_\alpha(z)|$, where $\rho_\alpha(z)$ is the one-sided stable PDF
with $0<\alpha<1$, analytically continued to the complex plane. Denote $z=re^{i\theta}$, where the modulus $r\in [0,1]$ and the angle $|\theta|<\theta_{*}(1-\alpha)$, with $0<\theta_{*}<\pi/6$ a critical angle to be determined later. The inverse Laplace transform gives
\begin{equation}
    \label{eq::ILT}
    \rho_\alpha(z)=\frac{1}{2\pi i}\int_{C}e^{\Phi_{z}(s)}\mathrm{d}s,
\end{equation}
where 
\begin{equation}
    \label{eq::phase_original}
    \Phi_{z}(s)=sz-s^{\alpha},
\end{equation}
 and $C$ is an appropriate inversion contour. To estimate $|\rho_\alpha(z)|$, we proceed in the following steps:
\begin{enumerate}
    \item Find the saddle point of $\Phi_z(s)$. Transform $C$ into the (approximated) steepest descent contour. Specifically, a parabolic contour derived from the local and global properties of $\Phi_z(s)$.
    \item Analyze properties of the phase function on the parabolic contour.
    \item Based on the results above, provide an upper bound of $|\rho_\alpha(z)|$.
\end{enumerate}
We proceed according to these steps in sequence.

\subsection{The parabolic contour}
From \cref{eq::phase_original}, we take derivative of $\Phi_z(s)$, then
\begin{equation}
    \label{eq::Phi_deri}
    \Phi_z'(s)=z-\alpha s^{\alpha-1},
\end{equation}
which shows the saddle point is $s_c=(\alpha/z)^{1/(1-\alpha)}$. The polar representation of $s_c$ is provided by $s_c=\sigma e^{i\psi_0}$, where
\begin{equation}
    \label{eq::sigma_psi}
    \sigma=\alpha^{1/(1-\alpha)}r^{-1/(1-\alpha)},\quad \psi_0=-\frac{\theta}{1-\alpha}.
\end{equation}
Thus one has $\Phi_z(s_c)=-(1-\alpha)s_c^{\alpha}$ and $\Phi_z''(s_c)=\lambda s_c^{\alpha-2}$, $\lambda=\alpha(1-\alpha)$. Based on these results, one selects the parameterized contour as $s(u)=s_cw(u)$, where
\begin{equation}
    \label{eq::sd_quad}
     w(u)=1+iu-\eta u^2,\quad \eta=\frac{2-\alpha}{6}\in(\frac{1}{6},\frac{1}{3}),\quad u\in\mathbb{R}.
\end{equation}
Correspondingly, the phase and the parameter derivative have the expression as
\begin{equation}
    \label{eq::phase_para}
\Phi_z(s(u))=\sigma^\alpha e^{i\phi} (\alpha w(u) - w(u)^\alpha), \quad
s'(u)=\sigma e^{i\psi_0}(i-2\eta u).
\end{equation}
where $\phi=\alpha\psi_0$. Hence the integral in \cref{eq::ILT} becomes
\begin{equation}
    \label{eq::integral_transform}
\rho_\alpha(z)=\frac{1}{2\pi i}\int_{-\infty}^{\infty}
\exp\Big(\sigma^\alpha e^{i\phi}(\alpha w(u) - w(u)^\alpha)\Big)
e^{i\psi_0}\sigma\bigl(i-2\eta u\bigr)\mathrm{d}u.
\end{equation}
Define the main phase function as
\begin{equation}
    \label{eq::main_phase}
    F(u):=\Re\{e^{i\phi}(\alpha w(u)-w(u)^{\alpha})\},
\end{equation}
and one readily derives the upper bound
\begin{equation}
    \label{eq::rho_bound}
    |\rho_\alpha(z)|\le\frac{1}{2\pi}\int_{-\infty}^{\infty}
\exp\Big(\sigma^\alpha F(u)\Big)\cdot 
\sigma(1+2\eta |u|)\mathrm{d}u.
\end{equation}

\subsection{The global maximum property of \texorpdfstring{$F(u)$}{F(u)}}
\label{sec::global_maximum}
For the main phase function $F(u)$ defined in \cref{eq::main_phase}, Proposition~\ref{prop::global_maximum} characterizes its most essential property.
\begin{proposition}[Preliminary bound on the global maximum]\label{prop::global_maximum}
Under suitable restrictions on the contour angle $\theta$, the main phase function $F(u)$ is monotonic on both sides of $u=0$, and $F(0)$ is its unique maximum. That is, $F'(u)>0$ for all $u<0$, and $F'(u)<0$ for all $u>0$.
\end{proposition}
By the symmetry \(F(-u;\phi)=F(u;-\phi)\), it suffices to prove
\(F'(u)<0\) for \(u>0\) and all admissible \(\phi\). The case \(u<0\)
then follows by replacing \(\phi\) with \(-\phi\). Before proving the global maximum property, we first consider the behavior of $F(u)$ as $|u|\rightarrow0$ and $|u|\rightarrow \infty$, described in Lemmas~\ref{lemma::u_small} and \ref{lemma::u_large}.

\begin{lemma}[Small-\(|u|\) Gaussian control]\label{lemma::u_small}
\begin{equation}
F(u)=-(1-\alpha)\cos\phi-\frac{\lambda}{2}\cos\phi u^2+O(u^4)\quad (|u|\to 0).
\end{equation}
Hence \(u=0\) is a strict local maximizer of \(F(u)\).
\end{lemma}

\begin{proof}
Let $G(u):=\alpha w(u)-w(u)^\alpha$. On the principal branch, we have
\be
G(0)=-(1-\alpha),\quad G'(0)=0,\quad
G''(0)=-\lambda,\quad G^{(3)}(0)=0.
\ee
The result follows from the Taylor expansion.
\end{proof}

\begin{lemma}[Large-\(|u|\) quadratic dominance]\label{lemma::u_large}
For sufficiently large \(|u|\),
\be
F(u)\le-\frac{\alpha \eta}{2}\cos\phi\cdot u^2.
\ee
\end{lemma}

\begin{proof}
This simply follows from the facts that $0<\alpha<1$, $\alpha w(u)$ dominates $w(u)^\alpha$ for large $|u|$,
and $-\eta u^2$ dominates in $w(u)$ for large $|u|$.
\end{proof}

Now we introduce the polar representation $w(u):=\varrho(u)e^{i\delta(u)}$, with $\varrho(u)=|w(u)|$ and $\delta(u)=\arg w(u)\in (0,\pi)$. Lemma~\ref{lem:delta-increasing} shows that the argument of $w(u)$ is monotonically increasing for $u>0$.

\begin{lemma}[Argument of $w(u)$ is increasing]\label{lem:delta-increasing}
\[
\delta'(u)=\frac{1+\eta u^2}{(1-\eta u^2)^2+u^2}>0,
\quad\forall u>0,
\]
and
\[
\delta\!\left(\frac{1}{\sqrt \eta}\right)=\frac{\pi}{2}.
\]
\end{lemma}

\begin{proof}
Choose the continuous branch of $\log$ along the path $u\mapsto w(u)$ and note
$\delta(u)=\mathrm{Im}\log w(u)$. Then
\[
\delta'(u)=\mathrm{Im}\left(\frac{w'(u)}{w(u)}\right)
=\mathrm{Im}\left(\frac{i-2\eta u}{1+i u-\eta u^2}\right)
=\frac{1+\eta u^2}{(1-\eta u^2)^2+u^2}>0,
\]
after multiplying numerator and denominator by $\overline{w(u)}$ and taking imaginary parts.

For the stated value at $u=1/\sqrt \eta$,
\[
w\left(\frac{1}{\sqrt \eta}\right)=1+i\frac{1}{\sqrt \eta}-\eta\cdot\frac{1}{\eta}
= i\frac{1}{\sqrt \eta},
\]
which is purely imaginary with positive imaginary part. Hence
\[
\delta(1/\sqrt \eta)=\arg\big(i/\sqrt \eta\big)=\pi/2.
\]
\end{proof}
Furthermore, with the polar representation of $w(u)$, the derivative of $F(u)$ is then
\begin{equation}    \label{eq::F_derivative}
F'(u)=\Re\Big\{\alpha e^{i\phi}(1-w(u)^{\alpha})w'(u)\Big\}.
\end{equation}
Hence we denote
\begin{equation}
    \label{eq::tangent}
    k(u)=2\eta u>0,\quad \tau(u)=\arctan(k(u))\in[0,\frac{\pi}{2}),\quad S(u)=\phi+\tau(u)
\end{equation}
as the positive tangent slope, tangent angle, and the transport angle of $w'(u)$, respectively, since $w'(u)=i-2\eta u$. We also introduce auxiliary exponents of modulus and argument that
\begin{equation}
    \label{eq::auxiliary}
\mu(u)=\varrho(u)^{\alpha-1},\quad \beta(u)=(1-\alpha)\delta(u)\in (0,\pi),
\end{equation}
which indicates that
\begin{equation}
    \label{eq::F_expression}
    \frac{F'(u)}{\alpha\sqrt{1+k(u)^2}}
=\mu(u)\sin\big(S(u)-\beta(u)\big)-\sin S(u).
\end{equation}
Subsequent proofs concerning monotonicity will rely on the expression provided in \cref{eq::F_expression}, requiring a detailed analysis of the ranges and quantitative relationships of $\mu(u)$, $S(u)$, and $\beta(u)$. Lemma~\ref{lem:rho-mu} provides the bound of $\varrho(u)$ and $\mu(u)$.

\begin{lemma}[Basic size bounds for $\varrho$ and $\mu$]\label{lem:rho-mu}
With $\varrho(u)=\sqrt{(1-\eta u^2)^2+u^2}$, one has
\[
\varrho(u)>1\quad\text{for all }u>0,\quad
\mu(u):=\varrho(u)^{\alpha-1}=\varrho(u)^{-(1-\alpha)}\in(0,1).
\]
\end{lemma}
\begin{proof}
Since
\[
\varrho(u)^2=1+(1-2\eta)\,u^2+\eta^2 u^4,\qquad \frac{\mathrm{d}}{\mathrm{d}u}\varrho(u)^2
=2u\bigl(1-2\eta+2\eta^2 u^2\bigr)>0\quad(u>0),
\]
and $\varrho(0)=1$, it follows that $\varrho(u)>1$ for $u>0$. Hence $\mu(u)\in(0,1)$.
\end{proof}
For the transport angle $S(u)$, Lemmas~\ref{lem:u0-delta-pi-over-2} and \ref{lem:second-quadrant} show the geometry of the case when $S(u)\ge \pi/2$.

\begin{lemma}[Geometry at $S=\pi/2$]\label{lem:u0-delta-pi-over-2}
Assume that
\begin{equation}
    \label{eq::phi_assume}
    0<\phi=\alpha\psi_0=-\frac{\alpha\theta}{1-\alpha}\le \alpha\theta_*.
\end{equation}
Then there exists a unique $u_0>0$ with $S(u_0)=\pi/2$, characterized by
\[
2\eta u_0=\cot\phi.
\]
Furthermore, under the standing constraints $\eta\in(1/6,1/3)$ and
$|\phi|\le \pi\alpha/6$, one has
\[
u_0\ge\frac{1}{\sqrt \eta},\quad \delta(u)\ge\frac{\pi}{2},\quad\forall u\ge u_0.
\]
\end{lemma}
\begin{proof}
Since $0<\phi\le \pi/6$, one has $\tan\phi\le \tan(\pi/6)=1/\sqrt3$. Note that $\eta\le 1/3$, one has
\[
\frac{1}{\sqrt3}\ \le\ \frac{1}{2\sqrt \eta}\quad\Longrightarrow\quad \tan\phi\le \frac{1}{2\sqrt \eta}\quad\Longrightarrow\quad \cot\phi\ge 2\sqrt \eta.
\]
Therefore,
\begin{equation}
    \label{eq::u0_estimate}
    u_0=\frac{\cot\phi}{2\eta}\ge \frac{2\sqrt \eta}{2\eta}=\frac{1}{\sqrt \eta}.
\end{equation}
By Lemma~\ref{lem:delta-increasing},
$\delta$ is increasing and $\delta(1/\sqrt \eta)=\pi/2$, so $\delta(u)\ge \pi/2$ for $u\ge u_0$.
\end{proof}

\begin{lemma}[Second quadrant when $S(u)>\pi/2$]\label{lem:second-quadrant}
If $u>0$ and $S(u)>\pi/2$, then $w(u)=1+iu-\eta u^2$ lies in the second quadrant. In particular,
\[
\delta(u) \in\ \Big(\frac{\pi}{2},\pi\Big),
\]
\end{lemma}

\begin{proof}
Since $S(u)=\phi+\tau(u)$ with $\tau(u)=\arctan k(u)$ strictly increasing in $k$, the condition
$S(u)>\pi/2$ is equivalent to
\[
\tau(u)>\frac{\pi}{2}-\phi
\quad\Longleftrightarrow\quad
k(u)>\cot\phi.
\]
Because $S(u)>\pi/2$ forces $\phi>0$ with $0<\phi\le \pi/6$, one has $\cot\phi\ge\sqrt{3}$.
Using $\eta\le1/3$, one gets $\sqrt{3}\ge 2\sqrt \eta$, hence
\[
k(u)>\cot\phi\ \ge\ 2\sqrt \eta \quad\Longrightarrow\quad u>1/\sqrt \eta.
\]
Since
\[
w(u)=(1-\eta u^2)+iu,
\]
one has $1-\eta u^2<0$ while $u>1/\sqrt \eta$. Thus $\Re \{w(u)\}<0$ and $\Im \{w(u)\}>0$, so $w(u)$
lies in the second quadrant and therefore $\delta(u)\in(\pi/2,\pi)$.
\end{proof}

Based on all the properties discussed above, we refine the angular constraints of Proposition~\ref{prop::global_maximum} and provide a proof that $F(u)$ has a unique global maximum at $F(0)$. The detailed description is given in Theorem~\ref{thm:Fprime-negative}.

\begin{theorem}[Strict monotonicity of the main phase $F(u)$]\label{thm:Fprime-negative}
Let $\cstar=1/9$ and $\thetastar=\arctan \cstar$.
Suppose that $|\theta|\le \thetastar(1-\alpha)$ and thus $|\phi|\le \alpha\thetastar$.
For $u>0$,
\[
F'(u)<0.
\]
Consequently $F$ is strictly decreasing on $(0,\infty)$, strictly increasing on $(-\infty,0)$, and attains its unique global maximum at $u=0$.
\end{theorem}
Proving Theorem~\ref{thm:Fprime-negative} requires a detailed discussion of the phase relationship between the two angles $S(u)$ and $\beta(u)$. Recall that $S(u)=\phi+\tau(u)\in[-\theta_{*},\pi/2+\theta_{*})$ and $\beta(u)=(1-\alpha)\delta(u)\in (0,\pi)$. For convenience, we split the whole phase diagram into four parts, which is clearly shown in \cref{fig::schema}. Specifically, the four cases are:
\begin{itemize}
    \item \textbf{Case A.} $S(u)< 0$.
    \item \textbf{Case B.} $0\le S(u)\le  \pi/2$.
    \item \textbf{Case C.} $S(u)>\pi/2$, and $\beta(u)\ge\pi/2$.
    \item \textbf{Case D.} $S(u)>\pi/2$, and $0<\beta(u)<\pi/2$.
\end{itemize}
\begin{figure}[!t]
    \centering
    \includegraphics[width=0.6\textwidth]{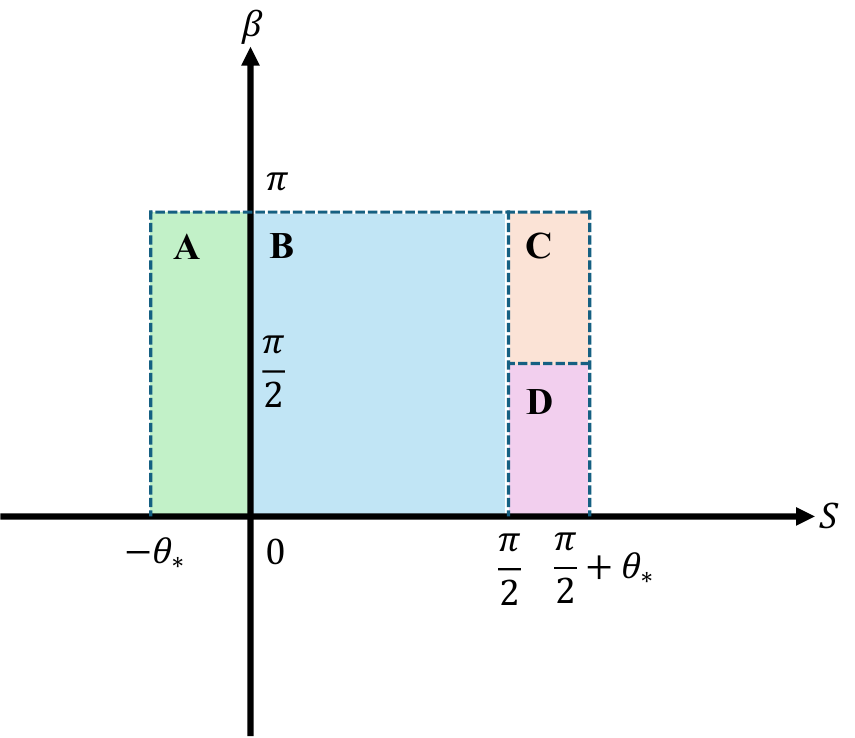}
    \caption{The $S$--$\beta$ phase diagram, divided into four regions, each shown in a different color.}
    \label{fig::schema}
\end{figure}

The following Propositions~\ref{prop::A}-\ref{prop::D} provide the whole proof process of Theorem~\ref{thm:Fprime-negative}.

\begin{proposition}[Case A]\label{prop::A}
If $S(u)<0$, then $F'(u)<0$.
\end{proposition}
\begin{proof}
    Since $S(u)=\phi+\tau(u)<0$ and $\tau(u)=\arctan(2\eta u)>0$, one must have $\phi<0$. Since $\tau(u)$ and $S(u)$ are strictly increasing, we denote the critical value
    $$u_{*}=\tan(-\phi)/(2\eta)>0$$
    such that $S(u_{*})=0$, so that Case A corresponds to $u\in [0,u_{*})$. Let $\zeta(u)=|\tan S(u)|=-\tan S(u)$, then one has the bound that
    \begin{equation}
    \label{eq::bound_tu}
        0\le \zeta(u)\le \tan(-\phi)\le \tan(\alpha\theta_{*})\le c_{*}=1/9,
    \end{equation}
    which also implies that
    \begin{equation}
    \label{eq::bound_u*}
    0<u_{*}\le c_{*}/(2\eta)<1/3    
    \end{equation}
    since $\eta\in (1/6,1/3)$. Rewrite \cref{eq::F_expression} as the form of
    \begin{equation}
         \frac{F'(u)}{\alpha\sqrt{1+k(u)^2}}=-\cos S(u)J(u),
    \end{equation}
    where 
    \begin{equation}
        \label{eq::J}
        J(u)=\mu(u)\sin\beta(u)-(1-\mu(u)\cos\beta(u))\zeta(u),
    \end{equation}
    and it suffices to show $J(u)>0$ for $u\in(0,1/3)$.

    Recall that $w(u)=(1-\eta u^2)+iu$, so on the negative-$S(u)$ interval $u\in(0,1/3)$ one has $1/2\le \Re\{w(u)\}\le 1$, which gives
    \begin{equation}   \label{eq::estimate_cos_delta}
        \cos \delta(u)=\frac{\Re\{w(u)\}}{\sqrt{\Re\{w(u)\}^2+u^2}}\ge \frac{1}{\sqrt{1+4u^2}}
    \end{equation}
    with
    \begin{equation}
        \label{eq::estimate_rho}
        \varrho(u)=\frac{\Re\{w(u)\}}{\cos \delta(u)}\le \sqrt{1+4u^2}.
    \end{equation}
    Hence the exponent $\mu(u)=\varrho(u)^{-(1-\alpha)}$ has the lower bound that
    \begin{equation}
    \label{eq::estimate_mu}
        \mu(u)\ge (1+4u^2)^{-(1-\alpha)/2}\ge 1-2(1-\alpha)u^2,
    \end{equation}
    where the last inequality is Bernoulli's inequality, valid since the exponent $-(1-\alpha)/2$ is negative. And for $\delta(u)$ itself, a simple estimate reads
    \begin{equation}
        \label{eq::estimate_delta}
        \arctan(u)\le\delta(u)=\arctan\left(\frac{u}{1-\eta u^2}\right)\le \arctan(2u),
    \end{equation}
    which implies the bound of $\beta(u)=(1-\alpha)\delta(u)$ that
    \begin{equation}
        \label{eq::estimate_sin_beta}
        \sin\beta(u)\ge \frac{2(1-\alpha)}{\pi}\arctan u\ge \frac{2(1-\alpha)}{\pi}\frac{u}{1+u^2}.
    \end{equation}
    and 
    \begin{equation}
        \label{eq::estimate_cos_beta}
        \cos\beta(u)\ge \cos((1-\alpha)\arctan(2u))\ge 1-2(1-\alpha)^2u^2
    \end{equation}
    with the well-known relation 
    \begin{equation}
        \frac{x}{1+x^2}\le \arctan x\le x,\quad x\in\mathbb{R}.
    \end{equation}

    Therefore, with the estimate of \cref{eq::estimate_mu,eq::estimate_sin_beta,eq::estimate_cos_beta}, one has that, for $u\in(0,u_{*})\subset(0,1/3)$,
    \begin{equation}
        \label{eq::estimate_J}
        \begin{aligned}
J(u)
&\ge \big(1-2(1-\alpha)u^2\big)\cdot
\frac{2}{\pi}(1-\alpha)\frac{u}{1+u^2}
-\big(4(1-\alpha)u^2\big)\cdot c_{*}\\
&=(1-\alpha)u\left[
\frac{2}{\pi(1+u^2)}\big(1-2(1-\alpha)u^2\big)-4u c_{*}
\right]\\
&>(1-\alpha)u\left[\frac{2}{\pi}\cdot\frac{9}{10}\cdot\frac{7}{9}-\frac{4}{27}\right]=(1-\alpha)u\left[\frac{7}{5\pi}-\frac{4}{27}\right]>0,\\
\end{aligned}
    \end{equation}
which guarantees $F'(u)<0$ for all $u$ for which the $S$--$\beta$ phase diagram lies in Case A.
\end{proof}

\begin{proposition}[Case B]\label{prop::B}
If $0\le S(u)\le\pi/2$, then $F'(u)<0$.
\end{proposition}
\begin{proof}
    One rewrites \cref{eq::F_expression} into
    \begin{equation}       \label{eq::F_caseB}
        \frac{F'(u)}{\alpha\sqrt{1+k(u)^2}}=(\mu(u)\cos\beta(u)-1)\sin S(u)-\mu(u)\cos S(u)\sin\beta(u).
    \end{equation}
    Here, at Case B with $S(u)\in [0,\pi/2]$, one has $\cos S\ge 0$, $\sin\beta>0$, $\cos\beta\le 1$ and $\mu\in(0,1)$, which imply that
    \begin{equation}
        \label{estimate_caseB}
        (\mu\cos\beta-1)\sin S\le(1-1)\sin S\le 0,\quad -\mu \cos S\sin \beta< 0.
    \end{equation}
    Hence $F'(u)<0$ holds when the $S$--$\beta$ phase diagram lies in Case B.
\end{proof}

\begin{proposition}[Case C]\label{prop::C}
If $S(u)>\pi/2$ and $\beta(u)\ge\pi/2$, then $F'(u)<0$.
\end{proposition}
\begin{proof}
    For this case, one must have $\phi>0$ by Lemma~\ref{lem:second-quadrant}. Note that $S(u)\le \phi+\pi/2$ with $\phi\in(0,\pi/6)$, so at Case C, one has
    \begin{equation}
        \label{eq::estimate_difference}
        -\frac{\pi}{2}<S(u)-\beta(u)\le \frac{\pi}{2}+\phi-\frac{\pi}{2}\le\frac{\pi}{6},
    \end{equation}
    which indicates
    \begin{equation}
    \label{eq::estimate_sin_diff}
      \sin(S(u)-\beta(u))\le 1/2   
    \end{equation}
    and also
    \begin{equation}
        \label{eq::estimate_sin_S}
        \sin S(u)\ge \sin(\frac{\pi}{2}+\phi)=\cos\phi\ge\frac{\sqrt{3}}{2}.
    \end{equation}
Substituting \cref{eq::estimate_sin_diff,eq::estimate_sin_S} into \cref{eq::F_expression}, as well as the relation $\mu\in(0,1)$,  one immediately derives
\begin{equation}   \label{eq::estimate_caseC}
 \frac{F'(u)}{\alpha\sqrt{1+k(u)^2}}=\mu\sin(S-\beta)-\sin S\le \frac{1}{2}-\frac{\sqrt{3}}{2}<0.    
\end{equation}
\end{proof}

\begin{proposition}[Case D]\label{prop::D}
If $S(u)>\pi/2$ and $0<\beta(u)< \pi/2$, then $F'(u)<0$.
\end{proposition}
\begin{proof}
    For this case, one must have $\phi>0$ by Lemma~\ref{lem:second-quadrant}. On the strip of $S(u)\in(\pi/2,\pi/2+\phi]$, one firstly has
    \begin{equation}
        \label{eq::estimate_S_caseD}
        \sin S\ge\sin\Big(\frac{\pi}{2}+\phi\Big)=\cos\phi,
\quad
\cos S\ge \cos\Big(\frac{\pi}{2}+\phi\Big)= -\sin\phi.
    \end{equation}
For the right-hand-side of \cref{eq::F_expression}, one has
\begin{equation}
    \label{eq::estimate_F_caseD}
    \begin{aligned}
         \mu\sin(S-\beta)-\sin S&=\sin S(\mu\cos\beta-1)-\mu\cos S\sin \beta\\
         &\le \cos \phi(\mu\cos\beta-1)+\mu\sin \phi\sin \beta\\
         &= \mu\cos(\beta-\phi)-\cos\phi.  
    \end{aligned}
\end{equation}
Hence it suffices to show that, for $u>0$ satisfies $S(u)\in (\pi/2,\pi/2+\phi]$ and $\beta\in (0,\pi/2)$,
\begin{equation}
    \label{eq::relation_caseD}
    T(u):=\mu(u)\cos(\beta(u)-\phi)<\cos \phi.
\end{equation}
Here we introduce the critical value $u_0=\cot\phi/(2\eta)$ in Lemma~\ref{lem:u0-delta-pi-over-2} such that $S(u_0)=\pi/2$, which forces $u>u_0$. At $u_0$, one has
\begin{equation}
    \label{eq::estimate_F_u0}
    \frac{F'(u_0)}{\alpha\sqrt{1+k(u_0)^2}}=\mu(u_0)\cos(\beta(u_0))-1<0,
\end{equation}
which guarantees the left edge of the valid interval of $u>u_0$.

    If $\beta(u)\ge 2\phi$, then
     $$\cos(\beta-\phi)-\cos\phi=2\sin\left(\frac{\beta}{2}\right)\sin\left(\phi-\frac{\beta}{2}\right)\le  0,$$
    which provides
    \begin{equation}
        T(u)\le \mu(u)\cos\phi<\cos\phi,
    \end{equation}
    satisfying \cref{eq::relation_caseD}.

    It remains to consider the case $0<\beta(u)<2\phi$, where $T(u)>0$ allows taking its logarithm. Since
$\mu=\varrho^{-(1-\alpha)}$ and $\beta=(1-\alpha)\delta$ by Lemma~\ref{lem:delta-increasing}, one has
\begin{equation}
   \log\frac{T(u)}{\cos\phi}
=
-(1-\alpha)\log\varrho
+
\log\frac{\cos(\beta-\phi)}{\cos\phi}. 
\end{equation}
Since $\beta-\phi\in(-\phi,\phi)$,
one therefore has the following integral estimate that
\begin{equation}
  \begin{aligned}
\log\frac{\cos(\beta-\phi)}{\cos\phi}
&=
\int_0^\beta
\frac{\mathrm{d}}{\mathrm{d}s}\log\cos(s-\phi)\,\mathrm{d}s  \\
&=
\int_0^\beta \tan(\phi-s)\,\mathrm{d}s  \\
&\le
\beta\tan\phi.
\end{aligned}  
\end{equation}
This indicates that
\begin{equation}
   \log\frac{T(u)}{\cos\phi}
\le
(1-\alpha)\left(\delta(u)\tan\phi-\log\varrho(u)\right), 
\end{equation}
and it suffices to prove
\begin{equation}
   \log\varrho(u)>\delta(u)\tan\phi. 
\end{equation}
Note that $2\eta u>\cot \phi$ at the case $S(u)>\pi/2$, then 
\begin{equation}
  u>u_0:=\frac{\cot \phi}{2\eta},  
\end{equation}
and thus $\varrho(u)>\varrho(u_0)$ by the strictly increase property of $\varrho$. Since $\phi\le \theta_{*}=\arctan(1/9)$ and $\eta\le 1/3$, one has
\begin{equation}
    \varrho(u_0)\ge \eta u_0^2-1=\frac{\cot^2\phi}{4\eta}-1\ge \frac{239}{4}. 
\end{equation}
Since $\delta(u)\in (0,\pi)$, one then has
\begin{equation}
\log \varrho(u)\ge \log \varrho(u_0)=\log\frac{239}{4}>\frac{\pi}{9}>\delta(u)\tan\phi,
\end{equation}
which provides
\begin{equation}
    \log\frac{T(u)}{\cos\phi}<0,\quad \forall u>u_0, 
\end{equation}
satisfying \cref{eq::relation_caseD}.
\end{proof}

\subsection{An upper bound of \texorpdfstring{$|\rho_\alpha(z)|$}{|rho\_alpha(z)|}}
\label{sec::upper_bound}
In \cref{sec::global_maximum}, we have shown that $F(0)$ becomes the strict global maximum of the main phase function $F(u)$. Here we provide an upper bound of $|\rho_\alpha(z)|$ based on estimate \cref{eq::rho_bound}, described in Theorem~\ref{thm:rho-corrected-fixed}.

\begin{theorem}[Explicit bound on $|\rho_\alpha(z)|$]\label{thm:rho-corrected-fixed}
For $z=re^{i\theta}$ with $r\in (0,1)$,
suppose that $|\theta|\le \theta_{*}(1-\alpha)$. Then one has,
\begin{equation}
    \label{eq::estimate_rho_general}
    |\rho_\alpha(z)|\ \le\ \frac{\sigma}{2\pi}e^{\sigma^\alpha F(0)}
\Bigg[
2R+2\eta R^{2}
+\Big(\frac{1}{aR}+\frac{2\eta}{a}\Big)e^{-aR^{2}}
\Bigg],
\end{equation}
where
\[
F(0)=-(1-\alpha)\cos\phi,\quad
a:=\sigma^\alpha\frac{\alpha \eta}{5}\cos\phi,
\]
and
\be
R=\max\left\{
\frac{20|\tan\phi|}{3\eta},\Big(\frac{20}{3\alpha \eta\cos\phi}\Big)^{\frac{1}{2-\alpha}},\Big(\frac{20\eta^{\alpha-1}}{3\alpha\cos\phi}\Big)^{\frac{1}{2(1-\alpha)}}
\right\}.\label{eq::R_def}
\ee
\end{theorem}

\begin{proof}
Recall the original expression of $F(u)=\Re\{e^{i\phi}(\alpha w(u)-w(u)^{\alpha})\}$ provided in \cref{eq::main_phase} with $w(u)=1+iu-\eta u^2$. A straightforward estimate gives
\be
F(u)\le \alpha\cos\phi-\alpha \eta\cos\phi u^2+\alpha|\sin\phi||u|+|u|^\alpha+\eta^\alpha|u|^{2\alpha}+1
\label{eq::F_bound}
\ee
which is derived via facts that $|w|\le 1+|u|+\eta u^2$ and $(x+y+z)^\alpha\le x^\alpha+y^\alpha+z^\alpha$, $\alpha\in (0,1)$.
Set 
$$c=\frac{1}{5}\alpha \eta\cos\phi.$$ 
For $|u|\ge R$, one requires simultaneously that
\[
\max\{\alpha|\sin\phi||u|, |u|^\alpha,\eta^\alpha|u|^{2\alpha},1\}
\le\frac{3}{20}\alpha \eta\cos\phi\cdot u^2,
\]
and 
\[
c|u|^2\ge\cos\phi,
\]
which leads to
\be
\ba
R &\ge \frac{20}{3}\,\frac{|\tan\phi|}{\eta}, \quad R \ge \Big(\frac{20}{3\,\alpha \eta\cos\phi}\Big)^{\!\frac{1}{2-\alpha}},\\
R&\ge \Big(\frac{20\,\eta^{\,\alpha-1}}{3\,\alpha\cos\phi}\Big)^{\!\frac{1}{2(1-\alpha)}}, \quad
R\ge \sqrt{\frac{20}{3\,\alpha \eta\cos\phi}},\quad
R\ge \sqrt{\frac{5}{\alpha \eta}}.
\ea
\ee
It is easy to see that
\be
\Big(\frac{20}{3\alpha \eta\cos\phi}\Big)^{\frac{1}{2-\alpha}}
\ge \sqrt{\frac{20}{3\alpha \eta\cos\phi}}
\ge \sqrt{\frac{5}{\alpha \eta}},
\ee
which provides the definition of scale $R$ in \cref{eq::R_def}. Then from \eqref{eq::F_bound}, one has that for $|u|\ge R$,
\begin{equation}
\begin{aligned}
F(u)&\le \alpha\cos\phi-\alpha \eta\cos\phi u^2
+\underbrace{\big(\alpha|\sin\phi|\,|u|+|u|^\alpha+\eta^\alpha|u|^{2\alpha}+1\big)}_{\le\ \frac{3}{5}\alpha \eta\cos\phi u^2}\\
&\le \alpha\cos\phi-\frac{2}{5}\alpha \eta\cos\phi u^2
\le \big(\alpha\cos\phi-\cos\phi\big)-\frac{1}{5}\alpha \eta\cos\phi u^2\\
&= F(0)-cu^2,
\end{aligned}\label{eq::F_estimate}
\end{equation}
where the last inequality follows from $c u^2\ge \cos\phi$.

One then splits the integral at $\pm R$ and bounds the Gaussian tails by the asymptotic analysis that
\begin{equation}
\label{eq::asymp_1}
    \int_R^\infty e^{-a u^2}\mathrm{d}u\le \frac{e^{-aR^2}}{2aR}
\end{equation}
and
\begin{equation}
\label{eq::asymp_u}
    \int_R^\infty u e^{-a u^2}\mathrm{d}u=\frac{e^{-aR^2}}{2a},
\end{equation} 
which yield the stated bound.
One sets
\[
a:=\sigma^\alpha c>0.
\]
Splitting the integral at $\pm R$ and using evenness of the bounds, one has
\[
\begin{aligned}
\int_{\mathbb{R}} e^{\sigma^\alpha F(u)}\bigl(1+2\eta|u|\bigr)\mathrm{d}u
\le e^{\sigma^\alpha F(0)}\Bigl[
&\int_{-R}^{R} (1+2\eta|u|)\mathrm{d}u\\
&{}+2\int_{R}^{\infty} e^{-a u^2}(1+2\eta u)\mathrm{d}u
\Bigr].
\end{aligned}
\]
The compact part is elementary, with
\begin{equation}
\label{eq::estimate_mid}
    \int_{-R}^{R} (1+2\eta|u|)\mathrm{d}u = 2\int_{0}^{R} (1+2\eta u)\mathrm{d}u=2R+2\eta R^2.
\end{equation}
For the tail integrals, the asymptotic estimates in \cref{eq::asymp_1,eq::asymp_u} give
\begin{equation}
\label{eq::estimate_tail}
2\int_{R}^{\infty} e^{-a u^2}(1+2\eta u)\mathrm{d}u\le \left(\frac{1}{aR}+\frac{2\eta}{a}\right)e^{-aR^2}.    
\end{equation}

Therefore, from \cref{eq::estimate_mid,eq::estimate_tail}, together with the estimate of $F(u)$ in \cref{eq::F_estimate}, one has
\[
\int_{\mathbb{R}} e^{\sigma^\alpha F(u)}\bigl(1+2\eta |u|\bigr)\,du
\ \le\ e^{\sigma^\alpha F(0)}
\left[\,2R+2\eta R^2+\left(\frac{1}{aR}+\frac{2\eta}{a}\right)e^{-aR^2}\right].
\]
Multiplying by the factor $\sigma/2\pi$ yields that
\[
|\rho_\alpha(z)|\le\frac{\sigma}{2\pi}e^{\sigma^\alpha F(0)}
\left[2R+2\eta R^2+\left(\frac{1}{aR}+\frac{2\eta}{a}\right)e^{-aR^2}\right],
\]
which provides an upper bound of $|\rho_\alpha(z)|$.
\end{proof}

Straightforward manipulations lead to Corollary~\ref{coro::theta_max}, which gives the estimate on the contour selected in \cref{sec::SOE_approx}.

\begin{corollary}
\label{coro::theta_max}
Take $|\theta|=(1-\alpha)\theta_{*}$.
Then the estimate becomes
\be
|\rho_\alpha(z)|
\le e^{-\Lambda(r;\alpha)}
\Bigg[
\frac{\sigma}{2\pi}\Big(2\widetilde R(\alpha)+\frac{2}{3}\widetilde R(\alpha)^2\Big)
+\ \frac{15}{\pi\cos \theta_{*}}\cdot\frac{1}{r}
\Big(\frac{1}{\widetilde R(\alpha)}+\frac{2}{3}\Big)
\Bigg],
\label{eq:rho-bound}
\ee
where
\be
\sigma=(\alpha r^{-1})^{1/(1-\alpha)},
\ee
\be
\Lambda(r;\alpha) = \sigma^{\alpha}(1-\alpha)\cos\theta_{*},  
\ee
\be
\widetilde R(\alpha)=
\max\left\{
\frac{40}{9},
\left(\frac{40}{\alpha\cos\theta_{*}}\right)^{\frac{1}{2-\alpha}},
\sqrt6\left(\frac{20}{3\alpha\cos\theta_{*}}\right)^{\frac{1}{2(1-\alpha)}}
\right\}.
\ee
\end{corollary}

\section{Integral bound of \texorpdfstring{$I_\alpha$}{I\_alpha}}
\label{sec::int_bound}
To estimate the error of SOE approximation in Theorem~\ref{thm::fractional_approx}, an upper bound of $I_\alpha$ is required with the expression of
\begin{equation}
    \label{eq::I_alpha_new}
    I_\alpha=\int_{0}^{\infty}|\rho_\alpha(re^{i\theta})|\mathrm{d}r,
\end{equation}
where $\theta=-(1-\alpha)\theta_{*}$.  The case
$\theta=(1-\alpha)\theta_{*}$ has the same value by the reflection property
$\rho_\alpha(\bar z)=\overline{\rho_\alpha(z)}$. \Cref{sec::rho_estimate}
provides estimates for $|\rho_\alpha(z)|$ within $|z|\le 1$. We therefore split
the integral into 
\begin{equation}
    \label{eq::I_alpha_split}
    I_\alpha=\int_{0}^{1}|\rho_\alpha(re^{i\theta})|\mathrm{d}r+\int_{1}^{\infty}|\rho_\alpha(re^{i\theta})|\mathrm{d}r:=I_\alpha^1+I_\alpha^2,
\end{equation}
and we provide the bound for each component.

For the compact part $I_\alpha^{1}$, the pointwise bound of
Corollary~\ref{coro::theta_max} is too conservative when $\alpha$ is close to
one.  It controls $|\rho_\alpha(z)|$ for each fixed $z$ by forcing all
lower-order terms in the phase to be dominated by the quadratic term, thereby
introducing the spurious factor $c^{1/(1-\alpha)}$.  For the integral defining
$I_\alpha^1$, a sharper route is to integrate first in $r$ and only then
estimate the saddle-contour integral.  We use the pointwise estimate only for
the complementary range $0<\alpha\le 1/2$, where its constants grow only
algebraically.

\begin{lemma}[Coercivity of the saddle phase near \(\alpha=1\)]
\label{lem::F_coercive_high_alpha}
Let $1/2\le\alpha<1$, and write the phase in \cref{eq::main_phase} as
\[
F_\alpha(u)=\Re\{e^{i\phi}(\alpha w(u)-w(u)^\alpha)\},\qquad
w(u)=1+iu-\eta u^2,\qquad \eta=\frac{2-\alpha}{6},
\]
with $|\phi|\le \alpha\theta_{*}$.  There is a universal constant $c_F>0$ such
that
\begin{equation}
\label{eq::F_coercive_high_alpha}
    -F_\alpha(u)\ge c_F(1-\alpha)(1+u^2),\qquad u\in\mathbb R .
\end{equation}
\end{lemma}

\begin{proof}
By Theorem~\ref{thm:Fprime-negative}, $F_\alpha$ attains its maximum at
$u=0$, and
\[
F_\alpha(0)=-(1-\alpha)\cos\phi .
\]
Hence, for any fixed $U_0>0$,
\[
-F_\alpha(u)\ge (1-\alpha)\cos\theta_{*}
\ge \frac{\cos\theta_{*}}{1+U_0^2}(1-\alpha)(1+u^2),
\qquad |u|\le U_0 .
\]
It remains to prove a quadratic lower bound for large $|u|$.  We take
$U_0=12$ and first consider $u\ge U_0$; the case $u\le -U_0$ follows from the
identity $F_\alpha(-u;\phi)=F_\alpha(u;-\phi)$.

Use the same polar representation as in \cref{sec::global_maximum},
$w(u)=\varrho(u)e^{i\delta(u)}$, where $\delta(u)\in(0,\pi)$, and set
$M(u)=\delta(u)+\phi$.  Since $1/2\le\alpha<1$, one has
$1/6<\eta\le1/4$.  For $u\ge U_0$,
\[
1-\eta u^2\le 1-\frac{u^2}{6}<0,\qquad
\delta(u)=\pi-\arctan\frac{u}{\eta u^2-1}
\ge \delta_0:=\pi-\arctan\frac{U_0}{U_0^2/6-1}.
\]
Here the last inequality follows from $\eta\ge1/6$ and the monotonicity of
$\delta(u)$ for $u>0$ in Lemma~\ref{lem:delta-increasing}.
Thus $M(u)\in[\delta_0-\theta_{*},\pi+\theta_{*}]$ and
\[
-\cos M(u)\ge c_M:=-\cos(\delta_0-\theta_{*})>0 .
\]
Indeed,
\[
\delta_0-\theta_{*}
=\pi-\left(\arctan\frac{12}{23}+\arctan\frac{1}{9}\right)
\in\left(\frac{\pi}{2},\pi\right),
\]
because $(12/23)(1/9)<1$. Hence $c_M>0$. For these fixed constants, a direct
calculation gives $c_M>0.83$ and $\pi\sin\theta_{*}<0.35$.

Using
$\varrho(u)^{-(1-\alpha)}=e^{-(1-\alpha)\log\varrho(u)}$, we rewrite the phase as
\begin{equation}
\label{eq::minus_F_polar}
\begin{aligned}
-F_\alpha(u)
&=\varrho(u)^\alpha\cos(M(u)-(1-\alpha)\delta(u))
-\alpha\varrho(u)\cos M(u)\\
&=\varrho(u)\Big[
\big(\alpha-\varrho(u)^{-(1-\alpha)}
\cos((1-\alpha)\delta(u))\big)(-\cos M(u))\\
&\hspace{3.8cm}
+\varrho(u)^{-(1-\alpha)}\sin((1-\alpha)\delta(u))\sin M(u)
\Big].
\end{aligned}
\end{equation}
Moreover, for $u\ge U_0$,
\[
\varrho(u)\ge \eta u^2-1\ge \frac{u^2}{7},\qquad \log\varrho(u)\ge 3 .
\]
The first coefficient in \cref{eq::minus_F_polar} satisfies
\[
\alpha-\varrho(u)^{-(1-\alpha)}\cos((1-\alpha)\delta(u))
\ge \alpha-e^{-3(1-\alpha)}\ge \frac{1-\alpha}{2}.
\]
The last inequality is the scalar bound
$1-s-e^{-3s}\ge s/2$ for $0\le s\le1/2$, applied with $s=1-\alpha$.

If $\sin M(u)\ge0$, the second term in \cref{eq::minus_F_polar} is nonnegative,
and therefore
\[
-F_\alpha(u)\ge \frac{c_M}{2}\,(1-\alpha)\,\varrho(u) .
\]
If $\sin M(u)<0$, then necessarily $M(u)\in(\pi,\pi+\theta_{*}]$, so
$|\sin M(u)|\le\sin\theta_{*}$.  Since
$0\le \sin((1-\alpha)\delta(u))\le (1-\alpha)\delta(u)\le (1-\alpha)\pi$,
\cref{eq::minus_F_polar} gives
\[
-F_\alpha(u)\ge (1-\alpha)\varrho(u)
\left(\frac{c_M}{2}-\pi\sin\theta_{*}\right).
\]
The fixed constant in parentheses is positive by the explicit bounds above.
Combining the two cases with $\varrho(u)\ge u^2/7$ gives
$-F_\alpha(u)\ge c(1-\alpha)u^2$ for $u\ge U_0$.  Together with the compact
estimate, this proves \cref{eq::F_coercive_high_alpha}.
\end{proof}

\begin{theorem}[Integral bound of $I_\alpha^{1}$]\label{thm:L1-rho}
There is a universal constant $C$ such that
\begin{equation}
\label{eq:rho-int-bound}
I_\alpha^1=\int_0^1|\rho_\alpha(re^{i\theta})|\,\mathrm{d}r
\le
\begin{cases}
C/\alpha, & 0<\alpha\le 1/2,\\[2mm]
C\log\!\bigl(e/(1-\alpha)\bigr), & 1/2\le\alpha<1,
\end{cases}
\end{equation}
where $|\theta|=(1-\alpha)\theta_{*}$.
\end{theorem}

\begin{proof}
For $0<\alpha\le1/2$, Corollary~\ref{coro::theta_max} implies the compact
pointwise bound \cref{eq:rho-bound}.  In this range, the factors in
$\widetilde R(\alpha)$ are uniformly bounded by $C\alpha^{-1/2}$, and
$\alpha^{-1/(1-\alpha)}\le C/\alpha$, while $e^{-B}\le1$.  Integrating
\cref{eq:rho-bound} over $r\in(0,1)$, using
\[
t=(1-\alpha)\cos\theta_{*}\,
\alpha^{\alpha/(1-\alpha)}r^{-\alpha/(1-\alpha)}
\]
for the exponentially decaying factor, gives
\[
I_\alpha^1\le C\left(\widetilde R+\widetilde R^2+
\alpha^{-1/(1-\alpha)}\Big(\frac1{\widetilde R}+1\Big)\right)
\le \frac{C}{\alpha}.
\]

It remains to treat $1/2\le\alpha<1$.  Starting from the saddle-contour bound
\cref{eq::rho_bound}, let
\[
\lambda=\sigma^\alpha=\alpha^{\alpha/(1-\alpha)}
r^{-\alpha/(1-\alpha)},\qquad
\lambda_0=\alpha^{\alpha/(1-\alpha)} .
\]
Then $\sigma\,\mathrm{d}r=-(1-\alpha)\,\mathrm{d}\lambda$, and
$\lambda_0\ge e^{-1}$ for $1/2\le\alpha<1$.  Since the integrand is
nonnegative, the Fubini--Tonelli theorem and $F_\alpha(u)<0$ give
\begin{equation}
\label{eq::I1_direct_integrated}
\begin{aligned}
I_\alpha^1
&\le \frac{1-\alpha}{2\pi}
\int_{-\infty}^{\infty}(1+2\eta|u|)
\int_{\lambda_0}^{\infty}e^{\lambda F_\alpha(u)}\,\mathrm{d}\lambda\,\mathrm{d}u\\
&= \frac{1-\alpha}{2\pi}
\int_{-\infty}^{\infty}(1+2\eta|u|)
\frac{e^{\lambda_0 F_\alpha(u)}}{-F_\alpha(u)}\,\mathrm{d}u .
\end{aligned}
\end{equation}
Using Lemma~\ref{lem::F_coercive_high_alpha}, $\eta\le1/3$, and
$\lambda_0\ge e^{-1}$, we obtain
\[
I_\alpha^1\le C\int_{-\infty}^{\infty}
\frac{1+|u|}{1+u^2}
\exp\{-c(1-\alpha)(1+u^2)\}\,\mathrm{d}u .
\]
The contribution from $|u|\le1$ is bounded by an absolute constant.  For
$|u|\ge1$,
\[
\int_1^\infty \frac{e^{-c(1-\alpha)u^2}}{u}\,\mathrm{d}u
=\frac12 E_1(c(1-\alpha))
\le C\log\!\left(\frac{e}{1-\alpha}\right),
\]
where $E_1(x)=\int_x^\infty e^{-t}t^{-1}\,\mathrm{d}t$ and the last inequality
is the standard small-$x$ bound on $E_1$.  This proves
\cref{eq:rho-int-bound}.
\end{proof}

For the infinite part $I_\alpha^{2}$, we shall use the asymptotic properties of $\rho_\alpha(z)$ as $|z|\ge 1$. Indeed, we have the following estimate on $|\rho_{\alpha}(z)|$ when $|z|>1$, as shown in Lemma~\ref{lemma::rho_alpha_asymp}.
\begin{lemma} 
\label{lemma::rho_alpha_asymp}
  For any complex number $z$ such that $|z| \ge 1$ and
  $|\arg(z)| < \frac{\pi}{2}(1-\alpha)$, the following inequality holds:
  \be 
  |\rho_\alpha(z)| \le \frac{C_\alpha}{|z|^{\alpha+1}},
  \label{eq:rhozbound1}
  \ee
  where the constant $C_\alpha$ is bounded by an explicit function of $\alpha$ that
  \be 
  C_\alpha \le \frac{1}{\pi} \left( \Gamma(\alpha+1) + \frac{1}{2^{1-\alpha}-1} \right). 
  \ee
\end{lemma}

\begin{proof} 
  The exact series representation for $\rho_\alpha(z)$ is given by
  \be 
  \rho_\alpha(z) = \frac{1}{\pi} \sum_{n=1}^{\infty} \frac{(-1)^{n-1}}{n!} \sin(\pi n \alpha) \Gamma(n\alpha+1) z^{-(n\alpha+1)}.
  \ee
  By applying the triangle inequality and the fact that $|\sin(\pi n \alpha)| \le 1$, we obtain a bound on the modulus that
  \be 
  |\rho_\alpha(z)| \le \frac{1}{\pi} \sum_{n=1}^{\infty} \frac{\Gamma(n\alpha+1)}{n!} |z|^{-(n\alpha+1)} 
  \ee
To analyze the behavior for $|z| \ge 1$, we factor out the dominant term $|z|^{-(\alpha+1)}$, providing that
  \be 
  |\rho_\alpha(z)| \le \frac{1}{|z|^{\alpha+1}} \left[ \frac{1}{\pi} \sum_{n=1}^{\infty} \frac{\Gamma(n\alpha+1)}{n!} |z|^{-(n-1)\alpha} \right].
  \ee
  Since $|z| \ge 1$ and $(n-1)\alpha > 0$ for $n \ge 2$, we have $|z|^{-(n-1)\alpha} \le 1$. We can therefore bound the series in the brackets by replacing $|z|$ with 1, such that
  \be 
  \frac{1}{\pi} \sum_{n=1}^{\infty} \frac{\Gamma(n\alpha+1)}{n!} |z|^{-(n-1)\alpha} \le \frac{1}{\pi} \sum_{n=1}^{\infty} \frac{\Gamma(n\alpha+1)}{n!}.
  \ee
  Let the constant on the right be $C_\alpha$. Our task is now to find an explicit bound for $C_\alpha$.
  We split the sum defining $C_\alpha$ at $n=1$ and derive
  \be 
  C_\alpha = \frac{1}{\pi} \left( \Gamma(\alpha+1) + \sum_{n=2}^{\infty} \frac{\Gamma(n\alpha+1)}{n!} \right).
  \ee
  Using the log-convexity of the Gamma function, we have the inequality $\Gamma(n\alpha+1) \le (n!)^\alpha$. For the factorial, we use the simple lower bound $n! \ge 2^{n-1}$ for $n \ge 2$. Applying these to the remainder sum gives
  \begin{equation}
  \begin{aligned}
  \sum_{n=2}^{\infty} \frac{\Gamma(n\alpha+1)}{n!}
  &\le \sum_{n=2}^{\infty} \frac{(n!)^\alpha}{n!}
  = \sum_{n=2}^{\infty} \frac{1}{(n!)^{1-\alpha}}\\
  &\le \sum_{n=2}^{\infty} \frac{1}{(2^{n-1})^{1-\alpha}}
  = \sum_{n=2}^{\infty} \left(\frac{1}{2^{1-\alpha}}\right)^{n-1}.
  \end{aligned}
  \end{equation}
  The final series is a geometric series with first term $a = 1/2^{1-\alpha}$ and ratio $r=1/2^{1-\alpha}$. Since $\alpha \in (0,1)$, the ratio is less than 1, and the series converges to $a/(1-r)$ such that
  \be 
  \sum_{n=2}^{\infty} \left(\frac{1}{2^{1-\alpha}}\right)^{n-1} = \frac{\frac{1}{2^{1-\alpha}}}{1 - \frac{1}{2^{1-\alpha}}} = \frac{1}{2^{1-\alpha}-1}.
  \ee
  Substituting this back gives the explicit bound for the constant
  \be 
  C_\alpha \le \frac{1}{\pi} \left( \Gamma(\alpha+1) + \frac{1}{2^{1-\alpha}-1} \right) 
  \ee
Combining the results from the previous steps, we arrive at the final inequality for $|z| \ge 1$,
  \be 
  |\rho_\alpha(z)| \le \frac{C_\alpha}{|z|^{\alpha+1}} \le \frac{1}{|z|^{\alpha+1}} \left[ \frac{1}{\pi} \left( \Gamma(\alpha+1) + \frac{1}{2^{1-\alpha}-1} \right) \right].
  \ee
\end{proof}

The following Corollary~\ref{coro::I_alpha_2} follows from integrating both sides of \eqref{eq:rhozbound1}, which provides the error bound of $I_{\alpha}^2$.
\begin{corollary}
\label{coro::I_alpha_2}
Suppose that $|\theta|<\frac{\pi}{2}(1-\alpha)$. Then
\begin{equation}
\label{eq::estimate_I_alpha_2}
I_{\alpha}^2=\int_1^{\infty}|\rho_{\alpha}(r e^{i\theta})|\,\mathrm{d}r
\le \frac{1}{\pi \alpha} \left( \Gamma(\alpha+1) + \frac{1}{2^{1-\alpha}-1} \right).
\end{equation}
\end{corollary}

Since $\Gamma(\alpha+1)$ is bounded on $(0,1)$ and
$2^{1-\alpha}-1\ge c(1-\alpha)$ for $0<\alpha<1$, \cref{eq::estimate_I_alpha_2}
implies
\[
I_\alpha^2\le C\left(\frac{1}{\alpha}+\frac{1}{1-\alpha}\right).
\]
Combining this estimate with Theorem~\ref{thm:L1-rho} proves
\cref{eq::I_alpha_upper}.

\begin{remark}
  \Cref{eq::estimate_I_alpha_2} shows singular behavior as $\alpha$ tends to
  $0$ or $1$. This is consistent with the limiting picture: as $\alpha\to1^-$
  the representing measure tends to a Dirac mass at $z=1$, while as
  $\alpha\to0^+$ the limiting Laplace transform is degenerate and is not
  represented by a regular probability density on $(0,\infty)$. Such singular
  limiting behavior is difficult to approximate using numerical methods, and the
  singularity in the error estimate reflects this phenomenon.

\end{remark}

\bibliographystyle{siamplain}
\bibliography{ref}

\end{document}